\documentclass[10pt]{amsart}
\usepackage{amsfonts,amssymb, latexsym}
\usepackage{amsmath}
\usepackage{cases}
\usepackage{amssymb}
\usepackage{graphicx}
\usepackage{color}
\usepackage{bigstrut}
\usepackage{mathtools}

\newtheorem{theorem}{Theorem}[section]
\newtheorem{proposition}[theorem]{Proposition}
\newtheorem{lemma}[theorem]{Lemma}
\newtheorem{corollary}[theorem]{Corollary}

\newtheorem{remark}[theorem]{Remark}

\theoremstyle{definition}%
\newtheorem{definition}[theorem]{Definition}

\def\pasdegrille{\let\grille = \pasgrille}

\def\aat#1#2#3{
\divide \dimen1 by 48 \dimen3=\dimen1 \multiply \dimen1 by #1
\advance \dimen1 by -\dimen3 \divide \dimen1 by 101 \multiply
\dimen1 by 100 \divide \dimen2 by \count11 \multiply \dimen2 by #2
\setbox0=\hbox{#3}\ht0=0pt\dp0=0pt
  \rlap{\kern\dimen1 \vbox to0pt{\kern-\dimen2\box0\vss}}\dimen1= \wd1
\dimen2=\ht1}
\def\pasgrille{
\count12= \dimen1 \divide \count12 by 50 \divide \dimen2 by \count12
\count11 =\dimen2 \ \divide \dimen1 by 48
\setlength{\unitlength}{\dimen1} \smash{\rlap{\ }} \dimen1= \wd1
\dimen2=\ht1 }
\def\grille{
\count12= \dimen1 \divide \count12 by 50 \divide \dimen2 by \count12
\count11 =\dimen2 \ \divide \dimen1 by 48
\setlength{\unitlength}{\dimen1}
\smash{\rlap{\graphpaper[1](0,0)(50, \count11)}} \dimen1= \wd1
\dimen2=\ht1 }

\pasdegrille

\numberwithin{equation}{section}

\newcommand{\R}{\mathbb{R}}

\newcommand{\N}{\mathbb N}

\newcommand{\beq}{\begin{equation}}
\newcommand{\eeq}{\end{equation}}
\newcommand{\ben}{\begin{eqnarray}}
\newcommand{\een}{\end{eqnarray}}
\newcommand{\beno}{\begin{eqnarray*}}
\newcommand{\eeno}{\end{eqnarray*}}

\newcommand{\indic}{1\!\!1}

\newcommand{\HHH}{\mathcal{H}}

\def\eps{\varepsilon}

\newcommand{\SF}{\mathcal H}

\title[]{Scattering to a stationary solution for the superquintic radial wave equation outside an obstacle}
\author{Thomas Duyckaerts${}^1$}
\address{Thomas Duyckaerts, LAGA, Institut Galil\'ee, Universit\'e Paris 13
99, avenue Jean-Baptiste Cl\'ement,n
93430 - Villetaneuse, France}
\email{duyckaer@math.univ-paris13.fr}

\author{Jianwei Yang${}^2$}
\address{Jianwei Yang, Department of Mathematics, Beijing Institute of Technology, Beijing 100081, P. R. China}
\email{jw-urbain.yang@bit.edu.cn}

\thanks{$^1$LAGA (UMR 7539), Universit\'e Paris 13, Sorbonne Paris Cit\'e, and Institut Universitaire de France}
\thanks{$^2$Department of Mathematics, Beijing Institute of Technology and LAGA (UMR 7539), Universit\'e Paris 13. Partially supported by the Labex MME-DII}

\date{\today}

\begin{document}


\begin{abstract}
We consider the focusing wave equation outside a ball of $\R^3$, with Dirichlet boundary condition and a superquintic power nonlinearity. We classify all radial stationary solutions, and prove that all radial global solutions are asymptotically the sum of a stationary solution and a radiation term.
\end{abstract}

\maketitle

\section{Introduction}

Let $K$ be a compact subset of $\R^3$ and $\Omega=\R^3\setminus K$. Consider a wave equation on $\Omega$ with Dirichlet boundary condition
\begin{equation}
\label{eq:gW}
\begin{cases}
(\partial_{t}^{2}-\Delta)u(t,x)=F(u),\,(t,x)\in\mathbb R\times \mathbb R^{3}\\
(u,\partial_{t}u)|_{t=0}=(u_{0}, u_{1}),
\quad u_{\restriction\partial \Omega}=0,
\end{cases}
\end{equation}
where the initial data $(u_0,u_1)$ is assumed to be in a Sobolev space, and in particular to have some decay at infinity. We will mainly be interested in a focusing supercritical nonlinearity $F(u)= |u|^{2m}u$, where $m>2$ is an integer, outside the unit ball of $\R^3$. We first review known results in more general cases.

The global dynamics of the linear wave equation ($F(u)=0$) is quite well understood, and depends on the geometry of the obstacle:
\begin{itemize}
\item When $K$ is non-trapping, for example convex, the global-in-time dispersive properties of the wave equation on the whole space $\R^3$ still hold. The local energy of smooth, compactly supported solutions decay exponentially (see \cite{MorawetzRalston77}. Strichartz estimates are available (see e.g. \cite{SmithSogge00}).
\item When $K$ is a trapping obstacle, some of the preceding properties persist, but it might be in weaker forms that depend on the geometry. In some weakly trapped geometries, the same Strichartz estimates as in $\R^3$ hold, as proved in \cite{Lafontaine18P}. In full generality,  the decay of the energy is only logarithmic (see \cite{Burq98}) and Strichartz estimates might hold only locally. 
\end{itemize} 
The defocusing equation $F(u)=-|u|^{2m}u$ was mainly considered in the energy-critical situation $m=2$ with a non-trapping obstacle. Once Strichartz estimates are known, the proof of global well-posedness can be easily adapted to this case (see \cite{SmithSogge95}). Under geometric assumptions that imply in particular that the obstacle is non-trapping, and are satisfied when $K$ is convex, it is proved in \cite{AbouShakra2013} that all solutions scatter to a solution of the linear wave equation (see also \cite{DuyckaertsLafontaine19P} for Neumann boundary conditions in a radial setting). This property persists in the super-critical case $m>2$ outside the unit ball, for radial solutions (see \cite{DAncona19P} and the Remark \ref{R:defocusing} below).

We are not aware of any work on focusing nonlinearity $F(u)=|u|^{2m}u$, except the recent preprint of P.~Bizo\'n and M.~Maliborski \cite{BizonMaliborski19P}. As in the case without obstacle, it is easy to construct, for any $m>0$,  solutions blowing up in finite time, using blow-up solutions of the ODE $y''=|y|^{2m}y$ and finite speed of propagation. 

We are interested in the behaviour of global solutions. The energy-critical case with $m=2$ on the whole space $\R^3$ was treated in a series of work initiated in \cite{KeMe08}. The equation has an explicit stationary solution $W(x)=(1+|x|^2/3)^{-1/2}$, which is unique up to scaling and change of sign. In \cite{DuKeMe13}, it is proved that any radial global solution of the equation is asymptotically 
the sum of decoupled rescaled stationary solutions and a solution of the free (linear) wave equation. When $m>1$ is not $2$, the global dynamics is 
different. There is no nonzero stationary solutions, and (assuming decay of the initial data) all known global solutions scatter to linear solutions. In particular, a solution whose Sobolev critical norm does not go to infinity scatters to a linear solution (see  \cite{DuKeMe12c}, \cite{Shen14}, \cite{DuyckaertsRoy17} and \cite{DuyckaertsYang18}, all concerning radial solutions). We note that the decay assumption of the initial data is necessary, as shows an example of J.~Krieger and W.~Schlag \cite{KriegerSchlag17}.

The purpose of this paper is to illustrate the fact that the obstacle might drastically change the dynamics of the focusing equation, even when the
dynamics of the linear and defocusing equations are essentially not modified by the presence of the obstacle. 

More precisely, we let $\mathbb{B}\subset \mathbb R^{3}$ be the unit ball centered at the origin, set $\Omega=\mathbb R^{3}\setminus \mathbb{B}$ and consider radial solutions of the equation \eqref{eq:gW} with $F(u)=|u|^{2m}u$, $m>2$.
\begin{equation}
\label{eq:NLW}
\begin{cases}
(\partial_{t}^{2}-\Delta)u(t,x)=|u|^{2m}u,\,(t,x)\in\mathbb R\times \Omega\\
(u,\partial_{t}u)|_{t=0}=(u_{0}, u_{1})\in \mathcal{H},\quad
u_{\restriction\partial \Omega}=0,
\end{cases}
\end{equation}
where $\HHH$ is the space of radial functions in $\dot{H}^1_0(\Omega) \times L^2(\Omega).$
One can prove that \eqref{eq:NLW} is locally well-posed in $\HHH$. 
The energy 
\begin{equation}
\label{eq:energy}
E(\vec{u}(t))=\frac12\int_{\Omega}\bigl|\nabla u(t,x)\bigr|^{2}dx
+\frac12\int_{\Omega}\bigl|\partial_{t}u(t,x)\bigr|^{2}dx
-\frac{1}{2(m+1)}\int_{\Omega}\bigl| u(t,x)\bigr|^{2(m+1)}dx\
\end{equation}
is conserved by the flow.
As mentioned before, the equation admits solutions blowing-up in finite time. 
More interestingly, there are also stationary solutions:
\begin{proposition}
\label{P:stationary}
 Assume $m>2$ is an integer. For any integer $k\geq 0$, there exists a unique radial stationary solution $Q_k\in C^{\infty}(\overline{\Omega})$ of \eqref{eq:NLW} such that $Q_k(x)=0$ for $x\in \partial \Omega$, and $r\mapsto Q_k(r)$ has exactly $k$ zeros on $(1,\infty)$, and is positive for large $r$. More precisely, there exists $c_k>0$ such that 
 $$\left|Q_k(r)-\frac{c_k}{r}\right|\lesssim \frac{1}{r^2},\quad \lim_{r\to\infty}r^2Q_k(r)=-c_k.$$
 Moreover, the sequence $(E(Q_k,0))_{k\in\mathbb{N}}$ is increasing.
 Finally the set of stationary solutions of \eqref{eq:NLW} is exactly 
 $$\left\{Q_k,\; k\in \N\right\}\cup \left\{-Q_k,\; k\in \N\right\}\cup \{0\}.$$
\end{proposition}
Our main result is that the stationary solutions $Q_k$ are the only obstruction to linear scattering for global solutions.  Consider the linear wave equation outside $\Omega$:
\begin{equation}
\label{eq:LW}
\begin{cases}
(\partial_{t}^{2}-\Delta)u(t,x)=0,\,(t,x)\in\mathbb R\times \Omega\\
(u,\partial_{t}u)|_{t=0}=(u_{0}, u_{1})\in \HHH,\quad 
u_{\restriction \partial \Omega}=0.
\end{cases}
\end{equation}
\begin{theorem}
\label{T:classification}
Let $u$ be a solution of \eqref{eq:NLW} on $[0,\infty)\times \Omega$. Then there exists a solution $v_L$ of the linear wave equation \eqref{eq:LW}, and a stationary solution $Q$ of \eqref{eq:NLW} such that
 $$\lim_{t\to\infty}\left\|\vec{u}(t)-\vec{v}_L(t)-(Q,0)\right\|_{\dot{H}^1(\Omega)\times L^2(\Omega)}=0.$$
 The same statement holds true for $t\to -\infty$.
\end{theorem}
According to Proposition \ref{P:stationary}, $Q$ must be $0$ (and in this case the solutions scatters to a linear solution) or one of the nonzero stationary solutions $\pm Q_k$. The set of initial data leading to scattering is open in $\HHH$. We conjecture that the set of data leading to  blow-up is open, and that the set of solutions converging locally to $\pm Q_k$ is a closed submanifold of $\HHH$, of codimension $k+1$ in $\HHH$. We will study this conjecture in a forthcoming paper. See \cite{BizonMaliborski19P} for numerical and analytical evidences toward this conjecture in the case $k=0$.

Note that Theorem \ref{T:classification} implies that for any $R>1$,
\begin{equation}
 \label{local_CV}
 \lim_{t\to\infty} \int_{1\leq |x|\leq R} \left|\nabla (u(t,x)-Q(x))\right|^2+(\partial_tu(t,x))^2\,dx=0.
\end{equation} 
An interesting question is the exact rate of this convergence when $Q=\pm Q_k$. This problem is discussed in \cite{BizonMaliborski19P} using both theoretical and numerical methods, in the case $k=0$.
Our method, based on a contradiction argument, does not give any quantitative information of this type. 
\begin{remark}
 \label{R:defocusing}
The proof of Theorem \ref{T:classification} can be adapted to prove that all solutions of the corresponding defocusing wave equation scatter to a linear solution (see  Remarks \ref{R:stationary_defocusing}, \ref{R:rigidity_defocusing} and \ref{R:proof_defocusing}). See also \cite{DAncona19P}, where a similar result is proved and used to treat nonradial perturbations of a radial solution.
 \end{remark}

The proof of Theorem \ref{T:classification} relies on the ``channels of energy'' method, which was introduced in \cite{DuKeMe11a}, and was used  in \cite{DuKeMe13} to prove the analog of Theorem \ref{T:classification} for the radial energy-critical wave equation in space dimension 3.  The proof for equation \eqref{eq:NLW} is somehow simpler, since equation \eqref{eq:NLW} does not admit any scaling invariance.
The core of the proof is the rigidity result (Proposition \ref{P:rigidity}) that states that any radial solution of  \eqref{eq:NLW} such that 
$$ \sum_{\pm}\lim_{t\to\pm\infty}\int_{|x|>|t|}|\nabla_{t,x}u(t,x)|^2\,dx=0$$
is stationary. This also implies the following one-pass theorem:
\begin{theorem}
 \label{T:one-pass}
Let $\eps>0$ be small and $k\in \N$. There exists $\delta>0$ with the following property. For all radial solution $u$ of \eqref{eq:NLW} such that there exists $t_0<t_1$ with $[t_0,t_1]\subset I_{\max}(u)$ and
$$ \left\|\vec{u}(t_0)-(Q_k,0)\right\|_{\HHH}\leq \delta,\quad \left\|\vec{u}(t_1)-(Q_k,0)\right\|_{\HHH}\geq \eps$$
one has
$$\forall t\in [t_1,+\infty)\cap I_{\max}(u),\; \forall Q\in \{0\}\cup \bigcup_{j\in \N}\{\pm Q_j\},\quad \left\|\vec{u}(t)-(Q,0)\right\|_{\HHH}> \delta.$$
\end{theorem}
This type of result is important to study the global dynamics of \eqref{eq:NLW} from a dynamical system point of view (see e.g. \cite{NaSc11Bo} for application of this type of one pass theorems in the context of nonlinear dispersive equations).

Our method also gives the classification of the dynamics below and at the ground state energy, in the spirit of \cite{KeMe08} and \cite{DuMe09b}.
By definition, the \emph{ground state} is the least energy nonzero stationary solution $Q_0$. The ground state and its opposite $-Q_0$ are the unique minimizers for the 
Sobolev type inequality: $\|f\|_{L^{2m+2}(\Omega)}\lesssim \|\nabla f\|_{L^2(\Omega)}$ (see Proposition \ref{prop:minimizer}).
As an immediate consequence of Theorem \ref{T:classification}, variational considerations and Proposition \ref{P:rigidity}, we obtain the classification of the dynamics below the energy of $Q_0$:
\begin{corollary}
 \label{cor:ground_state}
 Let $(u_0,u_1)\in\HHH$ such that $E(u_0,u_1)\leq E(Q_0,0)$, $u$ be the corresponding solution of \eqref{eq:NLW}, and $(T_-,T_+)$ the maximal interval of existence of $u$. 
 \begin{itemize}
 \item If $\int_{\Omega} |\nabla u_0|^2<\int_{\Omega} |\nabla Q_0|^2$, then $u$ is global, 
 $$\forall t\in \R,\quad \int_{\Omega}|\nabla u(t)|^2<\int_{\Omega} |\nabla Q_0|^2,$$
 and either $u$ scatters in both time directions, or $E(u_0,u_1)=E(Q_0,0)$ and there exists a sign $\pm$ such that $u$ scatters as $t\to \mp\infty$ and 
 \begin{equation}
  \label{CV_Q}
 \lim_{t\to \pm\infty}\|\vec{u}(t)-(Q_0,0)\|_{\HHH}=0.
 \end{equation} 
 \item If $\int_{\Omega}|\nabla u_0|^2=\int |\nabla Q_0|^2$, then $u$ is one of the two stationary solutions $\pm Q_0$.
 \item If $\int_{\Omega}|\nabla u_0|^2>\int_{\Omega}|\nabla Q_0|^2$, then 
 $$\forall t\in (T_-,T_+),\quad \int_{\Omega}|\nabla u(t)|^2>\int_{\Omega} |\nabla Q_0|^2.$$
 Furthermore, at most one of the times $T_+$ or $T_-$ is infinite. If $T_{\pm}$ is infinite for one sign $\pm$, then $E(u_0,u_1)=E(Q_0,0)$ and \eqref{CV_Q} is satisfied.
 \end{itemize}
\end{corollary}
In particular, if $E(u_0,u_1)<E(Q_0,0)$, there is an exact scattering/blow-up dichotomy, in the spirit of the articles of Kenig and Merle \cite{KeMe06,KeMe08} on critical Schr\"odinger and wave equations on $\R^N$. At the threshold energy, as in \cite{DuMe09a,DuMe09b}, a new type of solutions arise, satisfying \eqref{CV_Q} for one (and only one) sign $\pm$. As in \cite{DuMe09b}, one could prove the existence and uniqueness of such solutions, using the unique negative eigenvalue of the linearized operator at $Q_0$. We plan to treat these questions in a forthcoming paper.

Let us mention some related works. The defocusing energy-critical wave equation with a potential in dimension $3$ is considered in \cite{JiaLiuXu15,JiLiScXu17,JiLiScXu17P}. For this equation, there is no blow-up in finite time and every solution is global and scatters to a stationary solution, in the sense that the conclusion of Theorem \ref{T:classification} holds. The set of stationary solution for this equation is not classified as in Proposition \ref{P:stationary}, altough it is proved that for generic potential this set is finite. We refer to \cite{KeLaSc14} for the study of equivariant wave maps outside a ball. Again, there is no blow-up in finite time and every solution scatters to a stationary solution (an harmonic map), which is uniquely determined by the equivariance map of the equation. The underlying space dimension in \cite{KeLaSc14} is $5$, which makes the proofs more technically challenging, however the dynamics of equation \eqref{eq:NLW} is somehow richer, since blow-up in finite time is allowed, and there is a countable family of stationary solutions. In particular, one might contemplate solutions of \eqref{eq:NLW} that scatter to two distinct stationary solutions as $t\to+\infty$ and $t\to-\infty$.  

The outline of the paper is as follows. In Section \ref{S:preliminaries} we give some preliminaries on well-posedness (including a new profile decomposition for equation \eqref{eq:NLW}) and stationary solutions of \eqref{eq:NLW}. In Section \ref{S:classification} we prove our main result, the classification Theorem \ref{T:classification}. In Section \ref{S:further} we prove Corollary \ref{cor:ground_state} and \ref{T:one-pass}. Both proofs are short, relying on the rigidity Proposition \ref{P:rigidity}, and, for Corollary \ref{cor:ground_state}, on Theorem \ref{T:classification}.

\medskip

\subsection*{Notations}
If $a$ and $b$ are two positive quantities, we write $a\lesssim b$ when there exists a constant $C>0$ such that $a\leq C b$. 
We will write $a\approx b$ when we have both $a\lesssim b$ and $b\lesssim a$. We will write $a\ll b$ (resp. $a\gg b$) if there exists a sufficiently  large constant
$C>0$ such that $Ca\leq  b$ (resp. $a\geq Cb$).
We denote $\mathbb{N}$ the set of natural numbers.

We use $\mathbb B$ to denote the unit open ball
$\{x\in \mathbb R^{3}: |x|< 1\}$ and $\Omega=\R^3\setminus\mathbb{B}$.

The homogeneous Sobolev space $\dot{H}^1_0(\Omega)$ to be used frequently is defined as the closure of $C^{\infty}_{0}(\overline{\Omega})$ under the
$\dot H^{1}$ norm.  We refer to \cite{BlSmSo09,Burq03,Metcalfe04} for a systematic investigation on the homogeneous space $\dot{H}^s_D(\Omega)$ associated to the Laplacian $\Delta=\Delta_{D}$ subject to the Dirichlet boundary condition $u|_{\partial\Omega}=0$, with fractional $s$. We remark that  $\|f\|_{\dot{H}^1_0}\approx\|\sqrt{-\Delta}f\|_{L^2}$, where the latter norm is defined via the spectral resolution of $\Delta$.

For a radial function $f$ depending on $t$ and $r:=|x|$, we let $
\vec{f}:= (f,\partial_{t} f).
$
We let $L^p_t(I,L^q_x)$ be the space of measurable functions $f$ on $I\times \R^3$ such that
$$ \|f\|_{L^p_t(I,L^q_x(\Omega))}=\left(\int_I \left( \int_{\Omega}|f(t,x)|^qdx\right)^{\frac{p}{q}}\,dt  \right)^{1/p}<\infty.$$ 
For $q>1$, we use $q'=\frac{q}{q-1}$
to mean its Lebesgue conjugate. 
\\
\\
We denote by $S_{\rm L}(t)$ the linear propagator, \emph{i.e.}
\begin{equation*}
S_{\rm L}(t)(w_0,w_1) :=
\cos{(t\sqrt{-\Delta}_D)}w_0+\frac{\sin{(t\sqrt{-\Delta}_D)}}{\sqrt{-\Delta}_D}w_1.
\end{equation*}

\subsection*{Acknowledgement}
The first author would like to thank Piotr Bizo\'n for introducing equation \eqref{eq:NLW} and fruitful discussions on the subject.

\section{Preliminaries }
\label{S:preliminaries}
\subsection{Radial  linear wave solutions on $\Omega$}

Consider $u(t,x)$ a radial solution of \eqref{eq:LW}.
Assume that $(u_0,u_1)\in C^2(\R)$.
Using that $(\partial_t^2-\partial_r^2)(ru)=0$ and the boundary condition $u(t,1)=0$, we deduce that
\begin{equation}
\label{exp_u_lin}
ru=\psi(t+r)-\psi(t+2-r) 
\end{equation} 
for some function $\psi\in C^2(\R)$. One can compute $\psi$ using the initial condition:
\begin{equation}
\label{value_psi}
 \psi(\sigma)=\begin{cases}
\frac{1}{2}\Bigl[ \int_{1}^{2-\sigma} \rho u_1(\rho)\,d\rho-(2-\sigma)u_0(2-\sigma)\Bigr] ,&\text{ if }\sigma<1\\
       \frac{1}{2}\Bigl[   \int_{1}^{\sigma} \rho u_1(\rho)\,d\rho+\sigma\,u_0(\sigma)\Bigr] ,    &\text{ if }\sigma>1.
              \end{cases}
\end{equation}

and thus:
\begin{multline}
 \label{exp_u_lin_2}
 2ru(t,r)=\\\begin{cases} \int_{t+2-r}^{t+r} \rho u_1(\rho)\,d\rho +(r+t)u_0(r+t)-(t+2-r)u_0(t+2-r),&r-1<t,\\
        \int_{r-t}^{t+r} \rho u_1(\rho)\,d\rho +(r+t)u_0(r+t)+(r-t)u_0(r-t),&r-1>|t|\\
         \int_{r-t}^{2-t-r} \rho u_1(\rho)\,d\rho -(2-r-t)u_0(2-r-t)-(r-t)u_0(r-t),&r-1<-t.
        \end{cases}
\end{multline} 
We will also need the following exterior energy bound: 
\begin{lemma}
 \label{L:exterior_energy}
 Let $R\geq 1$, and $u$ be a radial solution of the linear wave equation \eqref{eq:LW} with initial data $(u_0,u_1)\in \HHH$. Then
 \begin{equation}
 \label{ext_bnd}
 \sum_{\pm}\lim_{t\to\pm\infty} \int_{R+|t|}^{\infty} (\partial_r(ru))^2+(\partial_t(ru))^2\,dr=\int_{R}^{+\infty} (\partial_r(ru_0))^2+r^2u_1^2\,dr.  
 \end{equation} 
\end{lemma}
\begin{proof}
By density, we can assume that $(u_0,u_1)$ is $C^2$.
 By explicit computation, and \eqref{exp_u_lin},
 $$(\partial_r(ru)^2+(\partial_t(ru))^2=2(\dot{\psi}^2(t+r)+\dot{\psi}^2(t+2-r)),$$
 and one can check that both sides of \eqref{ext_bnd} equal
 $$ 2\int_R^{\infty}\dot{\psi}^2+2\int_{-\infty}^{2-R} \dot{\psi}^2.$$
\end{proof}
\begin{remark}
 In the case $R=1$, we can check by integration by parts that
 $$\int_{1+|t|}^{+\infty}(\partial_r(ru))^2\,dr=\int_{1+|t|}^{+\infty}(\partial_ru)^2r^2\,dr+o(1),\quad t\to \pm \infty
 $$
 and
 $$ \int_1^{+\infty}(\partial_r(ru_0))^2\,dr=\int_1^{+\infty}(\partial_r u_0)^2r^2\,dr,$$
 and the preceeding lemma is equivalent to 
 $$\sum_{\pm}\lim_{t\to \pm\infty}\int_{|x|\geq 1+|t|}|\nabla u(t)|^2+(\partial_tu(t))^2\,dx=\int_{|x|\geq 1}|\nabla u_0|^2+u_1^2.$$
\end{remark}
The following asymptotics follow from \eqref{value_psi}
\begin{lemma}
 \label{L:asymptotic}
 For all $(u_0,u_1)\in \HHH$, we have, denoting by $u$ the solution of \eqref{eq:LW}
 \begin{equation}
  \label{Lpzero}\lim_{t \to \pm \infty} \int \frac{1}{|x|^2}|u(t,x)|^2dx+\|u(t)\|_{L^6(\Omega)\cap L^{\infty}(\Omega)}=0.
 \end{equation} 
 For both signs $+$ and $-$, there exists $G_{\pm}\in L^2(\R)$ such that 
 \begin{align}
  \label{as_dru}
  \lim_{t\to\pm \infty}\int_{1}^{\infty} \left|r\partial_ru(t,r)-G_{\pm}(r \mp t)\right|^2\,dr&=0\\
  \label{as_dtu}
  \lim_{t\to\pm \infty}\int_{1}^{\infty} \left|r\partial_tu(t,r)\pm G_{\pm}(r \mp t)\right|^2\,dr&=0.
 \end{align}
 Furthermore
 \begin{equation}
  \label{enG+-}
  \int_{\R} G_+^2(\eta)\,d\eta=
  \int_{\R} G_-^2(\eta)\,d\eta=\frac 12 \int_{1}^{+\infty} \left((\partial_tu(t,r))^2+(\partial_ru(t,r))^2\right)r^2dr, 
 \end{equation} 
 and both maps $(u_0,u_1)\to G_{\pm}$ are bijective.
\end{lemma}
\begin{proof}
 From the formula \eqref{value_psi}, we obtain \eqref{Lpzero}, as well as \eqref{as_dru} and \eqref{as_dtu} with $G_+(\sigma)=\psi'(2-\sigma)$, $G_-(\sigma)=\psi'(\sigma)$, that is:
 \begin{align*}
  G_+(\sigma)&=\frac 12\begin{cases}
-\sigma\,u_1(\sigma)+u_0(\sigma)+\sigma\,u_0'(\sigma)&\text{ if }\sigma>1\\
(2-\sigma) u_1(2-\sigma)+u_0(2-\sigma)+(2-\sigma)u_0'(2-\sigma)&\text{ if }\sigma<1
\end{cases}\\
G_-(\sigma)&=\frac 12\begin{cases}
\sigma\,u_1(\sigma)+u_0(\sigma)+\sigma\,u_0'(\sigma)&\text{ if }\sigma>1\\
-(2-\sigma) u_1(2-\sigma)+u_0(2-\sigma)+(2-\sigma)u_0'(2-\sigma)&\text{ if }\sigma<1.
\end{cases}
 \end{align*}
Note that $u_1\in L^2_{\rm rad}(\Omega)$, $u_0\in \dot{H}^1_{\rm rad}(\Omega)$ and Hardy's inequality imply $G_{\pm}\in L^2(\R)$ as announced. Using \eqref{as_dru}, \eqref{as_dtu} and the conservation of the energy for equation \eqref{eq:LW}, we obtain \eqref{enG+-}. It remains to prove that both maps $(u_0,u_1)\mapsto G_{\pm}$ are bijective. The injectivity follows immediately from \eqref{enG+-}. 

To prove the surjectivity, we let $G_+\in L^2(\R)$ (the proof is the same for $G_-$), and define, for $r>1$,
\begin{align*}
 u_0(r)&=\frac{1}{r} \int_1^{r}(G_+(\tau)+G_+(2-\tau))\,d\tau\\
 u_1(r)&=\frac{1}{r} (G_+(2-r)-G_+(r)).
\end{align*}
We notice that $(u_0,u_1)\in \dot{\SF}_{\rm rad}^1$. Indeed, since $G_+\in L^2(\R)$, we have
$$\int_1^{+\infty} (ru_1)^2\,dr<\infty, \quad \int_{1}^{+\infty}(\partial_r(ru_0))^2\,dr\leq 2\|G_+\|_{L^2}^2.$$
Furthermore, by a straightforward integration by parts,
$$\int_1^{R} u_0(r)\frac{d}{dr}(ru_0)\,dr=\frac{1}{2}\int_1^{R}u_0^2(r)\,dr+\frac{R}{2}u_0^2(R),$$
which shows by Cauchy-Schwarz that $\int_1^R |u_0(r)|^2\,dr\leq 4\|G_+\|_{L^2}^2$ for all $R>1$, and thus $\int_1^{+\infty} |u_0(r)|^2\,dr<\infty$. 

Letting $u$ be the solution of \eqref{eq:LW} with initial data $(u_0,u_1)$, we see from \eqref{value_psi} that $u$ satisfies \eqref{as_dru}, \eqref{as_dtu} (with the $+$ sign) which concludes the proof.
\end{proof}

\subsection{An overview of the Cauchy theory in $\mathcal{H}$}
\label{SS:Cauchy}
In this subsection, we recall  the local well-posedness theory of the problem 
\eqref{eq:NLW} in the energy space with 
radial initial data.\medskip

Let us start by recalling  the 
Strichartz estimate proved in \cite{SmithSogge00,Burq03,Metcalfe04}. 
\begin{proposition}
\label{P:Stricharz}
	Let $(q,r)$ such that $1/q+3/r=1/2$ and $q>2$. Then there exists $C_0>0$ such that, if 
	$u$ is a solution to the Cauchy-Dirichlet problem
	\begin{equation}
	\label{hdjsj}
	\begin{split}
	(\partial_{t}^2-\Delta)u(t,x)&=F(t,x), \quad(t,x)\in\R\times \Omega\\
	(u(0,x),\partial_{t}u(0,x))&=(u_0,u_1)\in\dot{H}^1_0(\Omega)\times L^2(\Omega)\\
	u(t,x)&=0, \quad x\in\partial\Omega.
	\end{split}
	\end{equation} 
	one has
\begin{equation}
\label{nsscbs}
\|u\|_{L^q_t(\R; L^r_x(\Omega))}\leq C_0\left(\|u_0\|_{\dot{H}^1_0(\Omega)}+\|u_1\|_{L^2(\Omega)}+\|F\|_{L^1_t(\R;L^2_x(\Omega))}\right).
\end{equation}
\end{proposition}
In the radial case, one can extend the range of Strichartz exponents, using the radial Sobolev inequality
\begin{equation}
 \label{rad_Sob}
 \forall R>1,\quad
 |f(R)|\lesssim \frac{1}{\sqrt{R}}\|f\|_{\dot{H}^1_{\rm rad}}.
\end{equation} 
Note that \eqref{rad_Sob} implies that for $6<p\leq \infty$, $\dot{H}^1_{\rm rad}(\Omega)$ is embedded into $L^p(\Omega)$ with compact embedding.
\begin{corollary}
\label{Co:Strichartz}
Assume that $1/q+3/r\leq 
1/2$ and $r$ is finite. There exists $C_0>0$ such that if $u$ and $F$ are radial solutions of \eqref{hdjsj}, then  \eqref{nsscbs} holds.
\end{corollary}
\begin{proof}
Assume that $\frac{1}{q}+\frac{3}{r}<\frac{1}{2}$, and let $q_1$ such that $\frac{1}{q_1}+\frac{3}{r}=\frac{1}{2}$. Since $r<\infty$, $q_1>2$.
 By energy inequalities and the embedding $H^1_{\rm rad}(\Omega)\subset L^{r}(\Omega)$, we have 
 $$\|u\|_{L^{\infty}_t(\R,L^{r}(\Omega))}\lesssim \|u_0\|_{\dot{H^1}(\Omega)}+\|u_1\|_{L^2(\Omega)}+\|F\|_{L^1_t(\R,L^2(\Omega))}.$$
 By standard Strichartz estimates,
 $$\|u\|_{L^{q_1}_t(\R,L^{r}(\Omega))}\lesssim \|u_0\|_{\dot{H^1}(\Omega)}+\|u_1\|_{L^2(\Omega)}+\|F\|_{L^1_t(\R,L^2(\Omega))},$$
and \eqref{nsscbs} follows since $q_2<q<\infty$.
\end{proof}
Note that the assumption $m>2$ implies that 
$q=2m+1$, $r=2(2m+1)$ satisfy the assumptions of Corollary \ref{Co:Strichartz}.

We state our main result in this subsection.

\begin{proposition}
	\label{pooiu}
	Assume $m\in(2,+\infty)\cap\mathbb{Z}$ in \eqref{eq:NLW}.
	Then for every  $\overrightarrow{{u}_0}:=(u_0,u_1)\in \HHH$, there exists a unique maximal  radial solution $u$ of \eqref{eq:NLW}
	defined in a maximal interval $[0,T^*)$ with
	$\overrightarrow{u}(0)=\overrightarrow{u_0}$
	and 
	$T^*\geqslant\overline{C}/\|\overrightarrow{u_0}\|_{\dot{\mathcal H}^1_{\rm rad}}^{2m}$,
	for  some universal constant $\overline{C}>0$,
	satisfying 
	\[
	u\in C\left([0,T^*),\dot{H}^1_0(\Omega)\right)\cap C^1\left([0,T^*),L^2(\Omega)\right).
	\]
	In addition, we have the following properties:
	\begin{enumerate}
		\item[(i)] either $T^*=+\infty$, or $T^*<+\infty$ and
		\begin{equation}
		\label{eq:nrm-contrl}
		\lim_{T\nearrow T^*}\|u\|_{L^{2m+1}_t\left([0,T],\;L^{2(2m+1)}_x(\Omega)\right)}=+\infty.
		\end{equation}
		Moreover, for every $T\in(0,T^*)$, the flow map $(v_0,v_1)\mapsto \vec{v}
		$ (where $v$ is the solution of \eqref{eq:NLW} with initial data $(v_0,v_1)$) is Lipschitz continuous from a 
		neighborhood of $(u_0,u_1)$ in $\HHH$ to $C\left([0,T],\HHH\right)$.
		\item[(ii)] If $\overrightarrow{u_0}\in \left(\dot{H}^2(\Omega)\cap \dot{H}^1_0(\Omega)\right)\times \dot{H}^{1}_0(\Omega)$ then $\overrightarrow{u}(t)\in 
		\left(\dot{H}^2(\Omega)\cap \dot{H}^1_0(\Omega)\right)\times \dot{H}^{1}_0(\Omega)$ for all $t\in (0,T^*)$.
		\item[(iii)] $E(\overrightarrow{u}(t))=E(\overrightarrow{u}_0)$ for every $t\in[0,T^*)$.
		\item[(iv)] If $\|u\|_{L^{2m+1}_t\left([0,T^*),L^{2(2m+1)}_x(\Omega)\right)}<\infty$, then $T^*=+\infty$ and $u$ scatters in the forward time direction, i.e.
		 there exists $\overrightarrow{u}_+=(u^+_0,u^+_1)\in\HHH$ such that
		\begin{equation}
		\label{habasbs}
		\lim_{t\to+\infty}\|\overrightarrow{u}(t)-\overrightarrow{S_{\rm L}}(t)\overrightarrow{u}_+\|_{\HHH}=0,
		\end{equation}
		where $\overrightarrow{S_{\rm L}}(t)\overrightarrow{u}_+:=\left(S_{\rm L}(t)\vec{u}_+,\partial_tS_{\rm L}(t)\vec{u}_+\right)$.
	    Conversely, if $u$ scatters, then 
		\begin{equation}
		\label{ncbsb}
			\|u\|_{L^{2m+1}_{t}\left( [0,\infty),\,L^{2(2m+1)}_{x}(\Omega)\right)}+\|\nabla_{x,t}u\|_{L^{\infty}_{t}\left([0,\infty),\,L_x^2(\Omega)\right)}<+\infty.
		\end{equation}
		\item[(v)] Let $I\subset[0,\infty)$ be a sub-interval
		such that
		\begin{equation}
		\label{eq:smallnorm}
			\|S_{\rm L}(t)\overrightarrow{u_0}\|_{L^{2m+1}_{t}\left( I,\,L^{2(2m+1)}_{x}(\Omega)\right)}=\delta
		\end{equation}
		with $0<\delta\ll1$ being sufficiently small.
		Then $u$ is defined on $I$.
		In particular, $I\subset [0,T^*)$ and moreover, with $C_0$ in \eqref{nsscbs}
		\[
		\lVert u-S_{\rm L}(t)\overrightarrow{u_0} \rVert_{L^{2m+1}_{t}\left( I,\,L^{2(2m+1)}_{x}(\Omega)\right)}\leqslant 
		\varepsilon:=C_0(2\delta)^{2m+1}
		\]
		The analogs of the statements (i)-(v) hold in the negative time direction as well.
	\end{enumerate}
\end{proposition}
The proof follows mainly from a standard 
fixed point argument based on Strichartz estimates in Corollary
\ref{Co:Strichartz} and we only sketch it here.
By using energy estimate, Sobolev embedding
and radial Sobolev inequality, it is readily to solve
\eqref{eq:NLW} on an interval $[0,T]$ with
$T<\bar{C} \|\overrightarrow{u_0}\|_{\HHH}^{-2m}$ for some constant $\bar{C}>0$,
depending only on $m, C_0$ and optimal constants in Sobolev embedding.
Let $T^*$ be the maximal time of existence.
Then (i)
follows by using Strichartz estimate.
(ii) is deduced from standard bootstrap argument based on the Duhamel formula.
By using (ii) and standard density argument,
we obtain the conservation of energy (iii). 
Finally, (iv) and (v) is immediately verified by
using Strichartz estimates and energy estimate.

 We next establish
 a long-time perturbation
lemma for \eqref{eq:NLW}.
\begin{lemma}
	\label{ksmssj}
	Given $M>0$, we have  $\varepsilon_M>0$
	and $C_M>0$ with the following properties.
	Let $I$ be an interval, $t_0\in I$, and $u,\tilde{u}\in L^{2m+1}\left(I,L^{2(2m+1)}_{\rm rad}(\Omega)\right)$ such that  
	$\vec{u},\vec{\tilde{u}} \in C(I,\HHH)$ 
    and
    \begin{equation}
    \label{nhss}
    \|\tilde{u}\|_{L^{2m+1}_t\left(I,L_x^{2(2m+1)}(\Omega)\right)}\le M,
    \end{equation}
    \begin{equation}
    \label{mnasa}
    \|{\rm eq}(u)\|_{L^1_t(I,L^2_x(\Omega))}+ \|{\rm eq}(\tilde{u})\|_{L^1_t(I,L^2_x(\Omega))}+\|R_{\rm L}\|_{L^{2m+1}_t\left(I,L^{2(2m+1)}_x(\Omega)\right)}= \varepsilon,
    \end{equation}
    where $\varepsilon\le\varepsilon_M$, ${\rm eq}(u)=(\partial_t^2-\Delta)u-|u|^{2m}u$ in the sense of distribution
    and $R_{\rm L}(t)=S_{\rm L}(t-t_0)(\vec{u}(t_0)-\vec{\tilde{u}}(t_0))$.
%
	Then 
	\[
	\|u-\tilde{u}\|_{L^{2m+1}_t(I,L^{2(2m+1)}_x(\Omega))}+\sup_{t\in I}\|\nabla_{x,t}(u(t)-\tilde{u}(t)-R_{\rm L}(t))\|_{L^2(\Omega)}\le C_M\varepsilon.
	\]
\end{lemma}
For the proof we will need the following Gr\"onwall-type lemma (see \cite{FaXiCa11})
\begin{lemma}
	\label{L:mnb}
	Let $1\leq \beta<\gamma\leq \infty$ and define $\rho\in [1,\infty)$ by $\frac{1}{\rho}=\frac{1}{\beta}-\frac{1}{\gamma}$.  Let $0<T\leq \infty$, $f\in L^{\rho}(0,T)$ and $\varphi\in L^{\gamma}_{\rm loc}([0,T))$ such that 
	$$\forall t\in [0,T),\quad \|\varphi\|_{L^{\gamma}(0,t)}\leq \eta+\|f\varphi\|_{L^{\beta}(0,t)}.$$
Then
$$\forall t\in [0,T), \quad \|\varphi\|_{L^{\gamma}(0,t)}\leq \eta\Phi\left(\|f\|_{L^{\rho}(0,t)}\right),$$
where $\Phi(s)=2\Gamma(3+2s)$ and $\Gamma$ is the Gamma function.
\end{lemma}

\begin{proof}[Proof of Lemma \ref{ksmssj}]
	Let $w:=u-\tilde{u}$. Then we have
	\[
	(\partial_t^2-\Delta)w=\underbrace{{\rm eq}(u)-{\rm eq}(\tilde{u})}_{:=e}+|u|^{2m}u-|\tilde{u}|^{2m}\tilde{u}.
	\]
	We assume to fix ideas $t_0=0$ and $I=[0,T)$.
By Duhamel's formula
	\begin{equation}
	\label{w_Duhamel}   
	w(t)=R_L(t)+\int_{0}^{t}\frac{\sin(t-s)\sqrt{-\Delta}}{\sqrt{-\Delta}}\left(e+(\tilde{u}+w)^{2m+1}-\tilde{u}^{2m+1}\right)(s)ds.
	  \end{equation} 
	Using Strichartz \eqref{nsscbs} and H\"older inequalities, we deduce
\begin{multline*}
 \forall t\in [0,T), \quad \|w\|_{L^{2m+1}((0,t)L^{2(2m+1)}(\Omega))}\leq \\
 \|R_L\|_{L^{2m+1}(0,t,L^{2(2m+1)}(\Omega))}
 +C_0\|e\|_{L^1((0,t),L^2(\Omega))}\\
 +C_0\int_0^t \left(\|w(\tau)\|_{L^{2(2m+1)}}\|\tilde{u}(\tau)\|^{2m}_{L^{2(2m+1)}(\Omega)}+\|w(\tau)\|^{2m+1}_{L^{2(2m+1)}(\Omega)}\right)d\tau.
\end{multline*}
From \eqref{mnasa}, we obtain
\begin{multline*}
 \forall t\in [0,T), \quad \|w\|_{L^{2m+1}((0,t)L^{2(2m+1)}(\Omega))}\leq \\
 (1+2C_0)\eps 
 +C_0\int_0^t \left(\|w(\tau)\|_{L^{2(2m+1)}}\|\tilde{u}(\tau)\|^{2m}_{L^{2(2m+1)}(\Omega)}+\|w(\tau)\|^{2m+1}_{L^{2(2m+1)}(\Omega)}\right)d\tau.
\end{multline*}
Let $\theta$ such that 
$$\|w\|_{L^{2m+1}((0,\theta)L^{2(2m+1)}(\Omega))}\leq C_M \eps$$
($C_M$ to be specified). Then
\begin{multline*}
 \forall t\in [0,\theta), \quad \|w\|_{L^{2m+1}((0,t)L^{2(2m+1)}(\Omega))}\leq \\
 (1+2C_0)\eps +C_0C_M^{2m+1}\eps^{2m+1}
 +C_0\int_0^t \left(\|w(\tau)\|_{L^{2(2m+1)}}\|\tilde{u}(\tau)\|^{2m}_{L^{2(2m+1)}(\Omega)}\right)d\tau\\
 \leq (2+2C_0)\eps +C_0\int_0^t \left(\|w(\tau)\|_{L^{2(2m+1)}(\Omega)}\|\tilde{u}(\tau)\|^{2m}_{L^{2(2m+1)}(\Omega)}\right)d\tau,
\end{multline*}
provided $C_0C_M^{2m+1}\eps^{2m}\leq 1$ (which holds if $\eps\leq \eps_M=1/C_0^{\frac{1}{m}}C_M^{2+\frac{1}{m}}$).

Using Lemma \ref{L:mnb} with
\begin{gather*}
\varphi(t)=\|w(t)\|_{L^{2(2m+1)}(\Omega)}, \quad f(t)=\|\tilde{u}(t)\|_{L^{2(2m+1)}(\Omega)}^{2m}\\
\beta=1,\quad \gamma=2m+1,\quad \rho=\frac{2m+1}{2m},
\end{gather*}
we obtain 
$$  \|w(t)\|_{L^{2m+1}\left((0,\theta),L^{2(2m+1)}(\Omega)\right)} 
\leq (2+2C_0)\Phi\left(C_0 M^{2m}  \right)\eps.
$$
Choosing $C_M>(2+2C_0)\Phi\left(C_0 M^{2m}  \right)$, we obtain by a simple bootstrap argument 
$$  \|w(t)\|_{L^{2m+1}\left((0,T),L^{2(2m+1)}(\Omega)\right)} 
\leq C_M\eps.
$$
The bound of $\|\nabla_{t,x}(u-\tilde{u}-R_L)\|_{L^2(\Omega)}$ follows from Strichartz estimates and the equality \eqref{w_Duhamel}.
\end{proof}
\begin{definition}
Let $\Sigma^+_{\rm rad}$ be the set of radial functions $(u_0,u_1)\in \HHH$
such that if $u$ is the solution of \eqref{eq:NLW}
with initial data $(u_0,u_1)$,
then $u(t,x)$ exists on $[0,+\infty)$ and scatters to a linear wave. We define $\Sigma_{\rm rad}^{-}$ similarly for the negative time direction.
\end{definition}
The following proposition is an immediate consequence of Lemma \ref{ksmssj} and the characterization of scattering from Proposition \ref{pooiu}:
\begin{proposition}
	\label{hdnsa}
	 $\Sigma_{\rm rad}^+$ and $\Sigma_{\rm rad}^-$ are open.
\end{proposition}
\subsection{Profile decomposition}
\label{SS:profile}
We prove here that there exists a profile decomposition which is adapted to the Strichartz norm used in the scattering theory of equation \eqref{eq:NLW}.
\begin{proposition}
 \label{P:profile}
 Let $(u_n)_n$ be a sequence of radial solutions of the linear wave equation outside the ball \eqref{eq:LW} such that $(\vec{u}_n(0))_n$ is bounded in $\HHH$. Then there exists a subsequence of $(u_n)_n$ (that we still denote by $(u_n)_n$), and, for any integer $j\geq 1$, a solution $U^j_L$ of \eqref{eq:LW} and a sequence $(t_{j,n})_n\in \R^{\N}$ satisfying
 $$
 j\neq j'\Longrightarrow\lim_{n\to\infty}
 \bigl| t_{j,n}-t_{j',n}\bigr|=+\infty,
 $$
  such that, letting, for $J\geq 1$,
 $$w_n^J(t)=u_n(t)-\sum_{j=1}^J U^j_L(t-t_{j,n}),$$
 we have, for all $(q,r)\in (2,\infty]\times (6,\infty)$ such that $\frac{1}{q}+\frac{3}{r}<\frac{1}{2}$,
 \begin{equation}
  \label{dispersive_part}
  \lim_{J\to\infty}\limsup_{n\to\infty}\|w_n^J\|_{L^q_tL^r_x}=0.
 \end{equation} 
 Furthermore,
 \begin{gather}
  \label{w-lim-profile}
  \forall j\geq 1,\quad \vec{u}_n(t_{j,n})\xrightharpoonup[n\to\infty]{} \vec{U}_L^j(0)\\
 \label{Pythagore}
 \forall J\geq 1, \quad \lim_{n\to\infty}\left\|\vec{u}_n(0)\right\|^2_{\SF}-\sum_{j=1}^J \left\|\vec{U}^j_L(0)\right\|^2_{\SF}-\left\|\vec{w}_n^J(0)\right\|_{\SF}^2=0.
 \end{gather} 
\end{proposition}
Proposition \ref{P:profile} is a consequence of the following Lemma:
\begin{lemma}
\label{L:profile}
 Let $(u_n)_n$ be a sequence of radial solutions of the linear wave equation on $\Omega$ \eqref{eq:LW} such that for all sequence $(t_n)_n\in \R^{\N}$, 
 $$\vec{u}_n(t_n)\xrightharpoonup[n\to\infty]{} (0,0) \text{ in }\HHH.$$
 Then for all $(q,r)\in (2,\infty]\times (6,\infty)$ such that $\frac{1}{q}+\frac{3}{r}<\frac{1}{2}$
 $$\lim_{n\to\infty} \|u_n\|_{L^q_tL^r_x}=0.$$
\end{lemma}
The fact that the Lemma implies Proposition \ref{P:profile} is by now standard (see e.g. the proof of Theorem 3.1 in \cite{DuyckaertsYang18}), and we omit it. 
\begin{proof}[Proof of Lemma \ref{L:profile}]
 We argue by contradiction. Assume that there exists a sequence of solutions $(u_n)_n$ of \eqref{eq:LW} such that for all sequence $(t_n)_n\in \R^{\N}$,
 \begin{equation}
 \label{weak_0}
  \vec{u}_n(t_n)\xrightharpoonup[n\to\infty]{} (0,0)\text{ in }\HHH.
 \end{equation} 
 Assume that there exist $(q,r)\in (2,\infty]\times (6,\infty)$ with $\frac{1}{q}+\frac{3}{r}<\frac{1}{2}$ and $\eps>0$ such that
 \begin{equation}
  \label{Str_lower_bnd}
 \forall n,\quad \left\|u_n\right\|_{L^q_tL^r_x}\geq \eps.
 \end{equation} 
 Let $(q_0,r_0)$ such that 
 \begin{equation*}
  \frac{1}{q_0}+\frac{3}{r_0}=\frac{1}{q}+\frac{3}{r},\quad 2<q_0<q,
 \end{equation*} 
and let $r_1$ such that $\frac{1}{q}+\frac{3}{r}=\frac{3}{r_1}$ (thus $6<r_1<\infty$). Then by H\"older's inequality,
$$\|u_n\|_{L^q_tL^r_x}\leq \|u_n\|^{\frac{q_0}{q}}_{L^{q_0}_tL^{r_0}_x}\|u_n\|_{L^{\infty}_tL^{r_1}_x}^{1-\frac{q_0}{q}}.$$
Since by Strichartz estimates $\|u_{n}\|_{L^{q_0}_tL^{r_0}_x}$ is bounded from above (see Corollary \ref{Co:Strichartz}), we deduce that there existe $\eps_1>0$ such that 
$$\forall n,\quad \|u_n\|_{L^{\infty}_tL^{r_1}_x}\geq \eps_1.$$
We thus can choose a sequence $(t_n)_n$ such that 
$$ \forall n,\quad \|u_n(t_n)\|_{L^{r_1}}\geq \frac{\eps_1}{2}.$$
This contradicts \eqref{weak_0} and the compactness of the embedding $\dot{H}^1_{\rm rad}(\Omega)\subset L^{r_1}(\Omega)$. The proof is complete.
 \end{proof}
 We will need to consider solutions to the wave equation \eqref{eq:NLW} outside wave cones. For this, it is convenient to multiply the nonlinearity by a characteristic function 
 \begin{definition}
If $(u_0,u_1)\in \dot H^{1}_{0}(\Omega)\times L^{2}(\Omega)$ and $R\geq 1$, 
\emph{the solution of \eqref{eq:NLW} on $\{|x|>R+|t|\}$, with initial data $(u_0,u_1)$}, is by definition the restriction to $\{|x|>R+|t|\}$ of the solution $u$ of the following wave equation,  
\begin{equation}
 \label{NLW_cone}
\begin{cases}
(\partial_{t}^{2}-\Delta)u(t,x)=F(t,x)\indic_{\{|x|>R+|t|\}},\,(t,x)\in\mathbb R\times \mathbb R^{3}\\
(u,\partial_{t}u)|_{t=0}=(u_{0}, u_{1}),\quad 
u|_{\partial \Omega}=0
\end{cases}
\end{equation}
where $F= \iota |u|^{2m}u$ with $\iota=\pm 1$, 
$m>2$. 
\end{definition}
One can adapt the well-posedness theory from Subsection \ref{SS:Cauchy}, yielding local well-posedness and maximal solution\footnote{Note however that since we have truncated the nonlinearity with a nonsmooth function, the persistence of regularity does not hold anymore} for equation \eqref{NLW_cone}. In particular, letting $T_R^*$ be the maximal time of existence for \eqref{NLW_cone}, we have the blow-up criterion
$$ T_R^*<\infty\Longrightarrow
\left\|u{\indic_{\{|x|>R+|t|\}}}  \right\|_{L^{2m+1}_t\left([0,T_R^*),\;L^{2(2m+1)}_x\right)}=\infty,$$
as well as the following scattering criterion. If
$$u {\indic_{\{|x|>R+|t|\}}}  \in {L^{2m+1}_t\left([0,+\infty),\;L^{2(2m+1)}_x\right)},$$
then $u$ scatters for positive times: there exists a solution $u_L$ of the linear wave equation on $\Omega$ such that 
$$ \lim_{t\to+\infty} \left\|\indic_{\{|x|>R+|t|\}}\left|\nabla_{t,x}u_L(t)-\nabla_{t,x}u(t)\right|\right\|_{\dot{H}^1(\Omega)\times L^2(\Omega)}=0.$$
Also, there exists $\eps_0>0$ (independent or $R>1$) such that if for some $T\in (0,\infty]$,
$$ \left\|S_L(t)(u_0,u_1)\indic_{\{|x|>R+|t|\}}  \right\|_{L^{2m+1}_t\left([0,T),\;L^{2(2m+1)}_x\right)}=\eps\leq \eps_0$$
then $T_R^*\geq T$ and
$$ \sup_{t\in [0,T]}\left\|\left|\nabla_{t,x}(u(t)-S_L(t)(u_0,u_1))\right|\indic_{\{|x|>R+|t|\}}\right\|_{L^2}\leq \eps^{2m+1}.$$
We note also that if $T^*$ is the maximal (positive) time of existence for the equation \eqref{eq:NLW} with the same initial data, then $T^*\leq T^*_R$ and the two solutions coincide on $\{(t,x), \; 0\leq t<T^*,\; |x|>R+|t|\}$.

Let $(u_{Ln})_n$ be a sequence of radial solutions of the linear wave equation \eqref{eq:LW} outside the ball. Assume that $(\vec{u}_n(0))_n$ is bounded in $\dot{H}^1(\Omega)\times L^2(\Omega)$ and has a profile decomposition $\{ U_L^j,(t_{j,n})\}_{j\geq 1}$ as in Proposition \ref{P:profile}. Extracting subsequences, reordering and time translating the profiles, we might assume
\begin{equation}
 \label{hyp_profiles}
 \forall n,\; t_{1,n}=0,\quad \forall j\geq 2,\; \lim_{n\to\infty} t_{j,n}\in \{\pm\infty\}.
\end{equation} 
We define the nonlinear profile $U^1$ associated to $U^1_L$ as the solution of the nonlinear wave equation \eqref{eq:NLW} with initial data $\vec{U}^1_L(0)$. If $R\geq 1$ we will also denote by $U^1$ the solution of \eqref{eq:NLW} on $\{|x|>R+|t|\}$ with the same initial data.
\begin{proposition}
\label{P:approx_profile}
Let $u_{Ln}$ be as above, and $R\geq 1$. Assume that the nonlinear profile $U^1$ is well-defined for $\{t\geq 0,\;|x|\geq R+|t|\}$, and  that
$$\indic_{\{|x|>R+|t|\}}U^1\in {L^{2m+1}\left((0,\infty),L^{2(2m+1)}\right)}.$$
Let $u_n$ be the solution of the nonlinear wave equation \eqref{eq:NLW} on $\{|x|>R+|t|\}$. Then for large $n$, $u_n$ is global for positive time, and, letting
$$ \epsilon_n^J(t,x)=u_n(t,x)-U^1(t,x)-\sum_{j=1}^J U^j_{L}(t-t_{j,n},x)-w_n^J(t,x),$$
one has
$$\lim_{J\to+\infty} \limsup_{n\to\infty} \sup_{t\geq 0} \int_{|x|>R+|t|}\left|\nabla_{t,x}\epsilon_n^J(t,x)\right|^2\,dx=0.$$
\end{proposition}
\begin{proof}
By Lemma \ref{ksmssj} (or rather its version adapted to solutions on $\{|x|>R+|t|\}$), it is sufficient to prove
$$ \lim_{J\to\infty}\limsup_{n\to\infty}\left\| \Big(\sum_{j=2}^J U^j_{L}(\cdot-t_{j,n})-w_n^J\Big)\indic_{\{|x|>R+|t|\}}\right\|_{L^{2m+1}_t\left((0,\infty),L^{2(2m+1)}_x\right)}=0.$$
Using that 
$$\lim_{J\to\infty}\limsup_{n\to\infty}\left\|w^J_n\right\|_{L^{2m+1}\left((0,\infty),L^{2(2m+1)}\right)}=0,$$
we see that it is sufficient to prove:
$$ J\geq 2\Longrightarrow \lim_{n\to\infty}\left\| \sum_{j=2}^J U^j_{L}(\cdot-t_{j,n})\indic_{\{|x|>R+|t|\}}\right\|_{L^{2m+1}_t\left((0,\infty),L^{2(2m+1)}_x\right)}=0.$$
Since $\lim_{n\to\infty}t_{j,n}\in \{\pm\infty\}$, this last property follows from the dominated convergence theorem, concluding the proof.
\end{proof}

\subsection{ Zeros of stationary solutions}
In this subsection, we state several properties on a class of singular stationary solutions  involved in 
\cite{ DuKeMe14,DuyckaertsRoy17,DuyckaertsYang18}.
\begin{proposition}
	\label{prop:DKM-stationary}
	Let $m>2$, $m\in\mathbb{N}$, and $\ell\in\mathbb R\setminus\{0\}$. 
	Then there exists a radial, $C^{2}$ solution $Z_\ell(x)=Z_\ell(|x|)$ of 
	\begin{equation}
	\label{eq:stationary}
	\Delta Z_{\ell}+Z_{\ell}^{2m+1}=0\quad\text{on}\quad
	\mathbb R^{3}\setminus\{0\},
	\end{equation}
	such that 
	\begin{equation}
	\label{eq:decay-infty}
	\forall\;r\geq 1,\quad
	\bigl | r \, Z_{\ell}(r)-\ell\bigr |\leq \frac{C}{r^{2}}
	\end{equation}
	\begin{equation}
	\label{eq:derivative-infty}
	\lim_{r\rightarrow\infty}
	r^{2}\frac{d Z_{\ell}}{dr}=-\ell\,.
	\end{equation}
	Furthermore, $Z_{\ell}\not\in L^{3m}$,
	where $3m$
	is the critical Sobolev exponent corresponding to 
	$\displaystyle s_{m}=\frac32-\frac{1}{m}$.
	In particular, $Z_{\ell}\not\in \dot{H}^{s_{m}}$.
	Moreover, the zeros of $Z_{\ell}$ are given by a sequence $\{r_{j}\}_{j=0}^{\infty}$ such that 
	\[
	r_{0}>r_{1}>\cdots>r_{j}>\cdots\longrightarrow 0,
	\quad j\rightarrow\infty.
	\]
\end{proposition}
\begin{remark}
	
	The existence of such a solution $Z_{\ell}$ with
	properties \eqref{eq:decay-infty}\eqref{eq:derivative-infty}
	and $Z_{\ell}\not\in L^{3m}$ had been demonstrated in
	\cite{DuKeMe14}.
	It remains to show that $Z_{\ell}(r)$ oscillates infinitely
	often towards $0$. This provides a more precise characterization on the behavior
	of $Z_{\ell}(r)$ as $r$ approaches the origin. 
\end{remark}
The proof of the oscillatiory property of $Z_\ell$ in Proposition \ref{prop:DKM-stationary} relies on the following classical result due to Fowler.
\begin{lemma}
	\label{jdhshs}
	Let $\theta(x)$ be a solution of 
	\begin{equation}
	\label{hsabss}
	\frac{d^2\theta}{dx^2}+x^{-4}\theta^n=0,\;x\in(0,+\infty),
	\end{equation}
	where $n>5$ is an odd integer. Then $\theta$ is one of the following three distinct types
	\begin{enumerate}
		\item [(i)]Special solutions 
		\begin{equation}
		\label{Q}
		\theta(x)=\pm\left(\frac{2(n-3)}{(n-1)^2}\right)^{\frac{1}{n-1}}x^{\frac{2}{n-1}};
		\end{equation}
		\item[(ii)] Emden's solutions with one arbitrary constant $C$
		\begin{equation}
		\label{nbssa}
		\theta(x)=C-\frac{\alpha(x) C^n}{6x^2},\quad \lim_{x\to\infty}\alpha(x)=1,
		\end{equation}
		\item[(iii)] $\theta(x)$ oscillates about $\theta=0$ with the asymptotic forms
		\begin{equation}
		\label{maas}
		\begin{split}
	|\theta(X_n)|&\approx AX_n^{\frac{4}{n+3}}\\
	x_{n+1}-x_{n}\approx&\frac{1}{A^{\frac{n-1}{2}}}
	\left(\frac{2}{n+1}\right)^{\frac{1}{2}}\frac{\Gamma\left(\frac{1}{2}\right)\Gamma\left(\frac{1}{n+1}\right)}{\Gamma\left(\frac{1}{2}+\frac{1}{n+1}\right)}X_{n}^{\frac{8}{n+3}}
	,	\end{split}
		\end{equation}
	where $A$ is a constant of integration, $\{X_n\}$ is the sequence of zeros of $\theta'(x)$, and
	$\{x_n\}$ is the sequence of zeros of $\theta(x)$, that satisfy $\lim_n x_n=+\infty$.
	\end{enumerate}
\end{lemma}	\begin{proof}Please see p. 281--282 of \cite{Fowler31}.
\end{proof}
\begin{remark}
	The equation \eqref{hsabss} along with its general form $\theta''+x^{\sigma}\theta^\lambda=0$ is usually referred as the Emden-Fowler equation.
	When $\lambda>1$ is not an integer, one may find in \cite{KiCh93BO} a similar classification on the solutions of Emden-Fowler equations in a more general setting.  
\end{remark}
\begin{proof}[Proof of the oscillation of $Z_\ell$]
	We may assume $\ell>0$ since the case $Z_{-\ell}=-Z_{\ell}$.
	By scaling invariance and the uniqueness of the fixed point argument,
	it suffices to consider $\ell=1$ 
	(see Remark 2.5 in \cite{DuyckaertsRoy17}) and
	we denote  by $Z(r)=Z_{1}(r)$ for brevity.
	\medskip
	
		Rewrite \eqref{eq:stationary} fulfilled by $Z$ 
	as the following ordinary differential equations (in the $r$ variables)
	\begin{equation}
	\label{eq:radial-ode}
	Z''(r)+\frac{2}{r}Z'(r)+Z(r)^{2m+1}=0\,.
	\end{equation}
	Let $h(s)=Z(1/s)$, $s\in (0,\infty)$.
	Then $h$ is a $C^{2}$ solution of 
	\begin{equation}
	\label{eq:h}
	h''(s)+s^{-4}h(s)^{2m+1}=0,\quad s>0\,,
	\end{equation}
	which satisfies 
	\begin{equation}
	\label{eq:h-aroundzero}
	\lim_{s\rightarrow 0} \frac{h(s)}{s}=1,\quad
	\lim_{s\rightarrow 0} h'(s)=1\,.
	\end{equation}
	We are reduced to showing that the zeros of $h$ form a sequence 
	$\{s_{j}\}_{j=0}^{\infty}$ such that
	\[
	0<s_{0}<s_{1}<s_{2}<\cdots<s_{j}<\cdots\longrightarrow\infty.
	\]
	In view of Lemma \ref{jdhshs}, it suffices to show that $h(s)$ is of type (iii).
	Invoking that $Z(r)$ is not bounded at the origin, 
	we see that $h(s)$ can not 
	be of the form \eqref{nbssa}.
   By \eqref{eq:h-aroundzero}, $h(s)$ is not a function given by the formula \eqref{Q}.
   %
Hence  $h(s)$ oscillates infinitely often and behaves  asymptotically according to formula \eqref{maas}. 
\end{proof}
\subsection{Radial stationary solutions outside the unit ball}
\label{SS:stationary}
Let $Z_{1}(x)$ be the radial solution of equation \eqref{eq:stationary} corresponding to $\ell=1$.
As we have seen in the last subsection, the zeros of $Z_{1}$
form a sequence 
$\{r_{j}\}_{j=0}^{\infty}$  with the following property
\begin{equation}
\label{hgnti}
r_{0}>r_{1}>\cdots>r_{j}>\cdots\longrightarrow 0,
\quad j\rightarrow\infty\;.
\end{equation}
Let $Q_{j}(r)=r_{j}^{1/m}Z_{1}(r_{j}r)$.
Then $Q_{j}(|x|)$ is the radial solution of the following elliptic equation outside the unit ball 
$\Omega=\mathbb R^{3}\setminus \mathbb B$
with the Dirichlet boundary condition
\begin{equation}
\label{eq:Q}
-\Delta Q=|Q|^{2m}Q,\quad Q|_{\partial\Omega}=0,\quad
x\in\Omega, \quad m>2\,,m\in\mathbb{N},
\end{equation}
where $\Delta=\Delta_D$ is the Dirichlet-Laplacian, and $Q$ belongs to $ \dot H^{1}_{0}(\Omega)$.

Notice that $Q_{j}(r)$
has exactly $j$ zeros in $(1,+\infty)$ for each $j\in \mathbb{N} $ and $Q_{j}(1)=0$.
Define the energy functional 
\[
E(Q)=\frac12\int_{\Omega}|\nabla Q(x)|^{2}dx
-\frac{1}{2(m+1)}\int_{\Omega} |Q(x)|^{2(m+1)}dx.
\]
Then 
\begin{equation}
\label{eq:scaled-energy}
E(Q_{j})=\frac{m}{2(m+1)}\,r_{j}^{-\frac{m-2}{m}}
\int_{r_{j}}^{+\infty}|Z_{1}(r)|^{2(m+1)}r^{2}dr.
\end{equation}
This formula with \eqref{hgnti} clearly yields  $E(Q_{j})\longrightarrow+\infty$ monotonically
as $ j\rightarrow \infty$. The following Lemma shows that there are no other stationary solutions for equation \eqref{eq:NLW}.
\begin{lemma}
\label{L:uniqueness}
 Let $Q\in \dot{H}^1(\Omega)$, radial, such that $-\Delta Q=|Q|^{2m}Q$. Then $Q\equiv 0$, or there exists a sign $\pm$ and $\alpha>0$ such that $Q(r)=\pm\alpha^{\frac 1m} Z_1\left( \alpha r \right)$. In particular, if $Q(1)=0$, then $Q\equiv 0$ or $Q=\pm Q_j$ for some $j\geq 0$.
\end{lemma}
\begin{proof}
 We first prove that there exists $\ell\in \R$ such that 
 \begin{equation}
\label{limQ}
\left|Q(r)-\frac{\ell}{r}\right|\lesssim \frac{1}{r^{2m-1}},\quad r\gg 1.
 \end{equation} 
 Indeed, we have $\frac{d^2}{dr^2}\left( rQ \right)=rQ^{2m+1}(r)$. Since by the radial Sobolev inequality \eqref{rad_Sob}, $|Q(r)|\lesssim 1/r^{1/2}$, we obtain that $\frac{d}{dr}(rQ)$ has a limit as $r\to \infty$. Using that 
 $\int_1^{\infty} \left|\frac{d}{dr}(rQ)\right|^2\,dr$ is finite, we see that this limit is $0$. Thus 
 \begin{equation}
  \label{T1}
  \frac{d}{dr}(rQ)=-\int_{r}^{\infty} \sigma Q^{2m+1}(\sigma)\,d\sigma.
 \end{equation} 
Combining with the radial Sobolev inequality, we obtain $\left|\frac{d}{dr}(rQ)\right|\lesssim \int_r^{\infty}\frac{1}{\sigma^{m-\frac 12}}\,d\sigma\lesssim \frac{1}{r^{m-\frac 32}}$. Since $m\geq 3$, we deduce that $rQ$ has a limit $\ell$. Plugging the estimate $|Q(r)|\lesssim 1/r$ into \eqref{T1} and integrating between $r$ and $\infty$, we obtain \eqref{limQ}.

If $\ell=0$, we let $Y(r)=0$ for $r>1$. If $\ell \neq 0$, we let $\alpha=|\ell|^{\frac{m}{1-m}}$, $\iota$ be the sign of $\ell$, and 
$$ Y(r)=\iota \alpha^{\frac{1}{m}} Z_1(\alpha r)$$
One can check
$$\lim_{r\to\infty} rY(r)=\ell.$$
We will prove that $Q\equiv Y$. Indeed, for large $r$
$$\left| \frac{d^2}{dr^2}(rQ-rY)\right|=r\left|Q^{2m+1}-Y^{2m+1}\right|\lesssim \frac{1}{r^{2m-1}}|Q(r)-Y(r)|.$$
Integrating twice, we deduce
\begin{multline*}
\left|rQ(r)-rY(r)\right|\lesssim \int_{r}^{+\infty}\int_{\sigma}^{+\infty} \rho^{-2m}\rho |Q(\rho)-Y(\rho)|\,d\rho\,d\sigma\\
\lesssim \frac{1}{r^{2m-2}}\sup_{\rho >r} \left| \rho\left( Q(\rho)-Y(\rho) \right) \right|.
\end{multline*}
Taking the supremum over all $r>R$, where $R\gg 1$ is fixed, we obtain that $Q(r)=Y(r)$ for large $r$. By classical ODE theory, we deduce that $Y(r)=Q(r)$ for all $r>1$.
 \end{proof}
\begin{remark}
\label{R:stationary_defocusing}
One can prove that the only stationary solution of the defocusing analog of \eqref{eq:NLW} (that is, with a minus sign in front of the nonlinearity) is $0$. More precisely, similarly to Proposition \ref{prop:DKM-stationary}  there is, for all $\ell \in \R\setminus \{0\}$, a solution $Z_{\ell}$ of the elliptic wave equation defined for large $r$ behaving as $\ell/r$ at infinity. However in this case, the solution $Z_{\ell}$ has a constant sign and is defined only for $r\in (R_{\ell},+\infty)$, for some minimal radius of existence $R_{\ell}>0$ that satisfies $\lim_{r\to R_{\ell}} |Z(r)|=\infty$ (see \cite[Proposition 2.3]{DuyckaertsRoy17}).  
\end{remark}

\begin{proposition}
	\label{prop:minimizer}
	For any radial $f\in \dot H^{1}_{0}(\Omega)$, we have
	\begin{equation}
	\label{eq:minimizer}
	\|f\|_{L^{2(m+1)}(\Omega)}\|\nabla Q_{0}\|_{L^{2}(\Omega)}
	\leq \|Q_{0}\|_{L^{2(m+1)}(\Omega)}\|\nabla f\|_{L^{2}(\Omega)}
	\end{equation}
	Furthermore, the equality is achieved in \eqref{eq:minimizer} if and only if there exists $\sigma \in \R$ such that $f=\sigma Q_0$.
\end{proposition}
\begin{proof}
	It suffices to show that if we set
	\[
	J(f)=\|\nabla f\|_{L^{2}(\Omega)}^{2(m+1)}\,/\,\|f\|_{L^{2(m+1)}(\Omega)}^{2(m+1)}, 
	\]
	and
	\[
	a=\inf\{ J(f):f\in \dot H^{1}_{0}(\Omega)\setminus \{0\}, f \;\text{radial}\},
	\]
	then $a=J(Q_{0})$. Notice that
	from radial Sobolev inequality, we have $0< a<+\infty$
	and hence the above two quantities are well-defined.
	\medskip
	
	The argument is reminiscent of \cite{Weinstein82}.
	Take a minimizing sequence $f_{\nu}\in \dot H^{1}_{0}(\Omega)$ which are radial such that $J(f_{\nu})\rightarrow a$ as $\nu\rightarrow+\infty$. Since $f_{\nu}$ is real valued, we may assume (replacing $f_{\nu}$ by $|f_{\nu}|$ if necessary), that$f_{\nu}$ is nonnegative. Setting $\varphi_{\nu}=f_{\nu}/\|f_{\nu}\|_{\dot H^{1}_{0}(\Omega)}$, we have $J(\varphi_{\nu})=J(f_{\nu})$ and
	$\|\nabla \varphi_{\nu}\|_{L^{2}(\Omega)}=1$. Hence there
	exists a subsequence $\varphi_{\nu_{k}}$ converges 
	weakly in $\dot H^{1}$ to $\varphi_{*}$ as $k\rightarrow +\infty$
	with $\|\varphi_{*}\|_{\dot H^{1}_{0}(\Omega)}\leq 1$.
	By using the radial Sobolev inequality and the Rellich-Kondrachov theorem, 
	one can show $\varphi_{\nu_{k}}$
	converges to $\varphi_{*}$ strongly in $L^{2(m+1)}(\Omega)$.
	As a consequence, $\varphi_{*}\neq 0$ since otherwise we would have $J(\varphi_{\nu_{k}})\rightarrow+\infty$ by the strong convergence. It follows from the above discussion that 
	\[
	a\leq J(\varphi_{*})\leq \frac{1}{\|\varphi_{*}\|_{L^{2(m+1)}(\Omega)}^{2(m+1)}}=\lim_{k\rightarrow+\infty}\frac{1}{\|\varphi_{\nu_{k}}\|_{L^{2(m+1)}(\Omega)}^{2(m+1)}}=a\,.
	\]
	Thus $J(\varphi_{*})=a$ and $\|\nabla \varphi_{*}\|_{L^{2}(\Omega)}=1$, which along with the weak convergence implies
	$\varphi_{\nu_{k}}\rightarrow \varphi_{*}$ in $\dot H^{1}_{0}(\Omega)$ strongly as $k\rightarrow+\infty$.
	\medskip
	
	It follows from the above facts that $\varphi_{*}$ is the 
	minimizer of the function $J$ and satisfies the 
	Euler-Lagrange equation:
	\[
	\frac{d}{d\eps}\Bigg|_{\eps=0}J(\varphi_{*}+\eps \eta)=0\,,\quad 
	\forall\;\eta\in C_{0}^{\infty}(\Omega).
	\]
	Taking $\|\nabla \varphi_{*}\|_{L^{2}(\Omega)}=1$  into account, we have
	\[
	-\Delta  \varphi_{*}=\frac{1}{\|\varphi_{*}\|_{2(m+1)}^{2(m+1)}}|\varphi_{*}|^{2m}\varphi_{*}.
	\]
	Let $\varphi_{*}(x)=\|\varphi_{*}\|_{L^{2(m+1)}}^{(m+1)/m}Q(x)$.
	Then we have $-\Delta Q=|Q|^{p-1}Q$ on $\Omega$ and $Q|_{\partial \Omega}=0$, $Q(x)\geq 0$ for $x\in \Omega$.
	By uniqueness of the solution for the problem \eqref{eq:Q} (Lemma \ref{L:uniqueness}), we have $Q(x)=Q_{0}(x)$. 
	
	Note that the last part of the argument above shows that any minimizer for $J$ is proportional to $Q_0$, which concludes the proof of the proposition.
\end{proof}

\section{Classification of global solutions}
\label{S:classification}
\subsection{Rigidity}
We prove here the following rigidity result:
\begin{proposition}
\label{P:rigidity}
 Let $\rho_0>1$ and $u$ be a solution of the nonlinear wave equation \eqref{eq:NLW} on $\{|x|>\rho_0+|t|\}$. Assume
 \begin{equation}
  \label{no_channel}
  \sum_{\pm} \lim_{t\to\pm\infty} \int_{\{|x|\geq |t|+\rho_0\}}|\nabla_{t,x}u(t,x)|^2\,dx=0.
 \end{equation} 
 Then $(u_0,u_1)(r)=0$ for almost all $r>\rho_0$, or there exists $\ell\in \R\setminus\{0\}$, $\iota\in \{\pm \}$ such that $(u_0,u_1)(r)=(\iota Z_{\ell}(r),0)$ for all $r>\rho_0$, where $Z_{\ell}$ is defined in Proposition \ref{prop:DKM-stationary}.
\end{proposition}
\begin{remark}
\label{R:rigidity_defocusing}
 Let us mention that the analog of Proposition \ref{P:rigidity}, with the same proof, is also valid for the defocusing equation corresponding to \eqref{eq:NLW}. In this case, in view of Remark \ref{R:stationary_defocusing}, the conclusion is that the solution $u$ is identically $0$.
\end{remark}

\begin{proof}
 The proof follows the line of the analogous result for the energy-critical wave equation on $\R^3$ (see \cite[Section 2]{DuKeMe13}), with some of the arguments simplified.  

 \emph{Step 1: channels of energy.} We fix a small $\eps>0$, and let $R\geq \rho_0$ such that
 \begin{equation}
  \label{OO30} \int_{|x|\geq R} |\nabla u_0|^2+u_1^2\,dx\leq \eps,
 \end{equation} 
 and prove, letting $v_0(r)=ru_0(r)$, $v_1(r)=ru_1(r)$,
 \begin{equation}
  \label{OO43}
  \int_R^{+\infty} (\partial_rv_0)^2+v_1^2\,dr\lesssim \frac{1}{R^{2m+1}}v_0^{2(2m+1)}(R).
 \end{equation} 
 Let $u_L$ be the solution of the linear wave equation with initial data $(u_0,u_1)$. We have (see Lemma \ref{L:exterior_energy}):
 $$\int_{R}^{+\infty} (\partial_r(ru_0))^2+(r u_1)^2\,dr\leq \sum_{\pm} \lim_{t\to\pm\infty} \int_{R+|t|}^{+\infty} \big(\partial_r(r u_L(t,r))\big)^2+r^2\big(\partial_t u_L(t,r)\big)^2\,dr.$$
 Furthermore, by the small data theory,
 $$\sup_{t\in \R}\big\|\indic_{\{|x|> |t|+R\}}\left(\nabla_{t,x}u(t)-\nabla_{t,x}u_L(t)\right)\big\|_{L^2}\lesssim \big\|(\nabla u_0,u_1)\indic_{\{|x|>R\}}\big\|^{2m+1}_{L^2}.$$
By a straightforward integration by parts, we have, for any $f_0\in \dot{H}^1_0(\Omega)$, and 
 \begin{equation}
  \label{OO41}
  \int_A^{+\infty} \left(\partial_r(rf_0)\right)^2\,dr=\int_A^{+\infty} (\partial_rf_0)^2r^2\,dr-Af_0^2(A),
 \end{equation} 
which yields, using assumption \eqref{no_channel}, 
$$\lim_{t\to\pm\infty} \int_{R+|t|}^{+\infty} \big(\partial_r(r u(t,r))\big)^2+r^2\big(\partial_t u(t,r)\big)^2\,dr=0.$$
 Combining, we obtain
\begin{equation}
 \label{OO40}
 \int_R^{+\infty} \big(\partial_r(ru_0)\big)^2+(ru_1)^2\,dr\lesssim \left(\int_{R}^{+\infty}\left(\left(\partial_r u_0\right)^2+u_1^2\right)\,r^2dr\right)^{2m+1}.
 \end{equation} 
 Using the formula \eqref{OO41} again, and the smallness assumption \eqref{OO30}, we deduce
 \begin{equation}
  \label{OO42}
  \int_{R}^{+\infty}\left(\partial_r(ru_0)\right)^2+(ru_1)^2\,dr\lesssim R^{2m+1}u_0^{2(2m+1)}(R),
 \end{equation} 
 hence \eqref{OO43}.
 
 \emph{Step 2: limit of $r\,u_0$.} In this step we prove that $v_0(R)$ has a limit $\ell$ as $R\to\infty$ and that there exists a constant $K$ (depending on $v$), such that
 \begin{equation}
  \label{limit_v}
  \left|v_0(R)-\ell\right|\leq \frac{K}{R^m},\quad \int_R^{+\infty} v_1^2(r)\,dr\leq \frac{K}{R^{2m+1}}.
 \end{equation} 
 Until the end of the proof, we will always denote by $K$ a large constant \emph{depending on $v$}, that may change from line to line. 
  
 We first fix $R,R'$ such that $\rho_0<R<R'<2R$ and the smallness assumption \eqref{OO30} is satisfied. Then
 $$|v_0(R)-v_0(R')|\lesssim\int_R^{R'} |\partial_r v_0(r)|\,dr\lesssim \sqrt{R}\sqrt{\int_R^{+\infty} (\partial_rv_0(r))^2\,dr}.$$
 Using Step 1, we deduce
 \begin{equation}
  \label{OO50}
|v_0(R)-v_0(R')|\lesssim \frac{1}{R^m}|v_0(R)|^{2m+1}.
 \end{equation} 
 By \eqref{OO30} and the integration by parts formula \eqref{OO41}, we have $\frac{1}{\sqrt{R}}|v_0(R)|\leq \sqrt{\eps}$, and thus
 \begin{equation}
  \label{OO51}
  |v_0(R)-v_0(R')|\lesssim \eps^m |v_0(R)|.
 \end{equation}
 By an easy induction, we deduce that for all $k\geq 0$, 
 $$ |v_0(2^k\rho_0)|\lesssim (1+C\eps^m)^k |v_0(\rho_0)|\leq K(1+C\eps^m)^k.$$
Going back to \eqref{OO50}, we obtain
$$ |v_0(2^k\rho_0)-v_0(2^{k+1}\rho_0)|\leq K 2^{-km}(1+C\eps^m)^{k(2m+1)}. $$
Taking $\eps>0$ small, we see that this implies that the series 
$\sum_{k\geq 0} |v_0(2^k\rho_0)-v_0(2^{k+1}\rho_0)|$ converges, and thus that there exists $\ell\in \R$ such that 
$$\lim_{k\to\infty} v_0(2^k\rho_0)=\ell.$$
This implies that $v_0(2^k\rho_0)$ is bounded. Using \eqref{OO50} again we obtain
$$ |v_0(2^k\rho_0)-v_0(2^{k+1}\rho_0)|\leq 2^{-km} K, $$
and summing up:
$$ \left|v_0(2^k\rho_0)-\ell\right|\leq K 2^{-km}.$$
By \eqref{OO50}, if $2^k\rho_0\leq r\leq 2^{k+1}\rho_0$,
$$ \left|v_0(2^k\rho_0)-v_0(r)\right|\leq K 2^{-km},$$
which concludes the proof of the first bound in \eqref{limit_v}. The second bound follows from \eqref{OO43}

\emph{Step 3. Compact support of the difference with a stationay solution.}
If $\ell\neq 0$, we let $Z_{\ell}$ be the radial solution of $-\Delta Z_{\ell}=Z_{\ell}^{2m+1}$ such that
\begin{equation}
 \label{OO60'}
 \left|Z_{\ell}(r)-\frac{\ell}{r}\right|\leq \frac{K}{r^3}.
\end{equation} 
(see Proposition \ref{prop:DKM-stationary}). We define $Z_0$ as the zero function, so that \eqref{OO60'} is also satisfied in the case $\ell=0$. Our goal is to prove that $(u_0,u_1)=(Z_{\ell},0)$ for almost every $r>\rho_0$. In this step, we prove that this equality holds for large $r$.

We let $h(r)=u-Z_{\ell}$, so that the following equation is satisfied for $r>\rho+|t|$
\begin{equation}
 \label{OO60}
 \partial_t^2h-\Delta h=(Z_{\ell}+h)^{2m+1}-Z_{\ell}^{2m+1}.
\end{equation} 
We let $(h_0,h_1)(r)=\vec{h}(0,r)$, and $h_L$ be the solution of the linear wave equation on $\{|x|>\rho_0+|t|\}$ with initial data $(h_0,h_1)$ at $t=0$. 

Let $R>\rho_0$ such that 
\begin{equation}
 \label{OO61}
 \sqrt{\int_R^{+\infty} \left( (\partial_r h_0)^2 +h_1^2 \right)r^2\,dr}+\left\| Z_{\ell}\indic_{\{r\geq R+|t|\}}\right\|_{L^{2m+1}_tL^{2(2m+1)}_x}\leq \eps,
\end{equation} 
where the small constant $\eps>0$ is to be specified later. Note that for any $\eps>0$, \eqref{OO61} is satisfied for large $R$. By the equation \eqref{OO60}, finite speed of propagation and Strichartz/energy estimates, for all interval $I$ containing $0$,
\begin{multline}
\label{OO70'}
 \sup_{t\in I}\left\| \left(\nabla_{t,x}h(t)-\nabla_{t,x}h_L(t)\right)\indic_{\{|x|>R+|t|\}}\right\|_{L^2_x}\\+\left\|(h-h_L)\indic_{\{|x>R+|t|\}}\right\|_{L^{2m+1}_t(I,L^{2(2m+1)})}\\
 \lesssim \left\|\indic_{\{|x>R+|t|\}} |Z_{\ell}|^{2m} |h|\right\|_{L^1_t(I,L^2_x)}+\left\|\indic_{\{|x>R+|t|\}} |h|^{2m+1}\right\|_{L^1_t(I,L^2_x)}, 
\end{multline} 
and thus, by H\"older's inequality, and the bound of the norm of $Z_{\ell}$ in \eqref{OO61}, we deduce
\begin{multline}
\label{OO70}
 \sup_{t\in I}\left\| \left(\nabla_{t,x}h(t)-\nabla_{t,x}h_L(t)\right)\indic_{\{|x|>R+|t|\}}\right\|_{L^2_x}\\
 +\left\|(h-h_L)\indic_{\{|x>R+|t|\}}\right\|_{L^{2m+1}_t(I,L^{2(2m+1)})}\\
 \lesssim \eps^{2m}\left\|\indic_{\{|x>R+|t|\}} h\right\|_{L^{2m+1}_t\left(I,L^{2(2m+1)}_x\right)}+\left\|\indic_{\{|x>R+|t|\}} h\right\|^{2m+1}_{L^{2m+1}_t\left(I,L^{2(2m+1)}_x\right)}.
 \end{multline} 
 Combining with the smallness assumption on $h$ in \eqref{OO61}, we deduce
 \begin{equation*}
  \sup_{t\in I}\left\| \left(\nabla_{t,x}h(t)-\nabla_{t,x}h_L(t)\right)\indic_{\{|x|>R+|t|\}}\right\|_{L^2_x}\lesssim \eps^{2m}\|(\nabla h_0,h_1)\indic_{\{|x|>R\}}\|_{L^2_x}.
 \end{equation*} 
 By the same argument as in Step 1, we obtain 
 \begin{equation}
  \label{OO71}
  \int_R^{+\infty} \left(\left( \partial_r(rh_0) \right)^2+r^2h_1^2(r)\right)\,dr\lesssim \eps^{4m}R h_0^2(R).
 \end{equation} 
 Arguing as in Step 2, we deduce that for $R<R'<2R$, if \eqref{OO61} holds, one has
 $$|g_0(R)-g_0(R')|\lesssim \eps^{2m}|g_0(R)|,$$
 where $g_0(R)=Rh_0(R)$.
 By a straightforward induction argument, we deduce
 \begin{equation}
  \label{OO80}
  |g_0(R)|\lesssim \frac{1}{\left( 1-C\eps^{2m} \right)^k} \left|g_0\left( 2^kR\right)\right|.
 \end{equation} 
 However, by Step 2 and \eqref{OO60'}, there exists a constant $K$ such that 
 $$ |g_0(2^kR)|\leq \frac{K}{(2^kR)^{2}}.$$
 Taking $\eps$ small, so that $1-C\eps^{2m}>\frac{1}{2}$, we deduce from \eqref{OO80} that $Rh_0(R)=g_0(R)=0$,  if \eqref{OO61} is satisfied, that is for large $R$. Going back to \eqref{OO71} we obtain that $h_1(R)=0$ for almost all large $R$. This concludes this step noting that $(h_0,h_1)=(u_0,u_1)-(Z_{\ell},0)$.
 
 \emph{Step 4. End of the proof.} We next prove that $(u_0,u_1)=(Z_{\ell},0)$ for almost every $r>\rho_0$. We let 
 $$\rho_1=\inf \left\{ \rho>\rho_0\;\Big|\; \int_{\rho}^{+\infty} \left((\partial_rh_0)^2+h_1^2\right)r^2\,dr=0\right\}.$$
 We must prove that $\rho_1=\rho_0$. We argue by contradiction, assuming that $\rho_1>\rho_0$. We thus can choose $R$ such that $\rho_0<R<\rho_1$ and
\begin{equation}
 \label{OO81}
 \int_R^{+\infty} \left( (\partial_r h_0)^2 +h_1^2 \right)r^2\,dr+\left\| Z_{\ell}\indic_{\{R+|t|\leq r\leq \rho_1+|t|\}}\right\|_{L^{2m+1}_tL^{2(2m+1)}_x}\leq \eps.
\end{equation} 
 By finite speed of propagation and the definition of $\rho_1$, $r\leq \rho_1+|t|$ on the support of $h$. As a consequence, we see that the argument of Step 3 is still valid, replacing $\indic_{\{r\geq R+|t|\}}$ by $\indic_{\{R+|t|\leq r\leq \rho_1+|t|\}}$ in \eqref{OO70'}. In particular, $h_0(R)=0$, and \eqref{OO71} holds for this choice of $R$. This implies
 $$\int_{R}^{+\infty} \left((\partial_rh_0)^2+h_1^2\right)r^2\,dr=0,$$
 contradicting the definition of $\rho_1$.
\end{proof}
\subsection{Boundedness along a sequence of times}
\begin{lemma}
\label{L:bnd}
 Let $u$ be a solution of \eqref{eq:NLW} such that $T_+(u)=+\infty$. Then 
 $$\liminf_{t\to+\infty} \int_{\Omega} |\nabla u|^2+(\partial_tu)^2\,dx\leq \frac{4(m+1)}{2m}E(u_0,u_1).$$
 In particular, $E(u_0,u_1)>0$ or $(u_0,u_1)=0$.
\end{lemma}
\begin{proof}
 The proof is very close to the one of the analogous result in the energy-critical case without obstacle (see \cite[Prop 3.4]{DuKeMe13}). It uses a monotonicity formula that goes back to the work of Levine \cite{Levine74}. We argue by contradiction, assuming that $E(u_0,u_1)<0$, or that there exists $t_0>0$ such that for all $t\geq t_0$,
 \begin{equation}
  \label{OO90}(1-\eps_0)\left( \|\nabla u(t)\|^2_{L^2(\Omega)}+\|\partial_t u(t)\|^2_{L^2(\Omega)} \right)\geq \left( \frac{4(m+1)}{2m} \right)E(u_0,u_1)+\eps_0.
 \end{equation} 
We let $\varphi\in C^{\infty}(\R^3)$ be a radial function such that $\varphi(r)=1$ if $r\leq 2$ and $\varphi(r)=0$ if $r\geq 3$. We let 
$$ y(t)=\int_{\Omega} \varphi\left( \frac{x}{t} \right)u^2(t,x)\,dx.$$
We will prove that there exists $\gamma>1$ and $t_1\geq t_0$ such that
\begin{gather}
 \label{OO101}
 \forall t\geq t_1,\quad \gamma y'(t)^2\leq y(t)y''(t)\\
\label{OO102}
\forall t\geq t_1,\quad y'(t)>0,
 \end{gather} 
 yielding a contradiction by a standard ODE argument (see e.g. the end of the proof of Theorem 3.7 in \cite{KeMe08} for the details).

 Using the small data theory and finite speed of propagation, we obtain that 
 $$\lim_{t\to\infty} \int_{\{|x|>\frac{3}{2}|t|\}} |\nabla_{t,x}u|^2+\frac{1}{|x|^2}|u|^2+|u|^{2m+2}\,dx=0.$$
As a consequence, using also equation \eqref{eq:NLW} and integration by parts, we obtain, as $t\to\infty$:
\begin{gather}
 \label{OO103}
 y'(t)=2\int_{\left\{1\leq |x|\leq 2t\right\}} u\partial_tu\,dx +o(t)\\
 \label{OO104}
 y''(t)=2\int_{\Omega} (\partial_tu)^2-2\int_{\Omega} |\nabla u|^2 +2\int_{\Omega} |u|^{2m+2}+o(1).
\end{gather}
We can rewrite \eqref{OO104}:
\begin{equation}
 \label{OO105}
y''(t)=2m\int_{\Omega} |\nabla u|^2+(2m+4)\int_{\Omega} (\partial_tu)^2-4(m+1)E(u_0,u_1)+o(1).
 \end{equation} 
Using that $E(u_0,u_1)<0$ or that \eqref{OO90} holds, we deduce that there exists $\eps_0>0$ such that for large $t$, $y''(t)\geq \eps_0.$
This yields
$$ \liminf_{t\to\infty} \frac{1}{t}y'(t)\geq \eps_0.$$
In particular \eqref{OO102} holds. More precisely, for large $t$,
$\int_{\{1\leq |x|\leq 2t\}} u\partial_tu\geq \frac{\eps_0}{2}t$, and \eqref{OO103} implies 
$$y'(t)\leq (2+o(1))\int_{\{1\leq |x|\leq 2t\}} u\partial_tu.$$
By \eqref{OO105} and the fact that $E(u_0,u_1)$ is negative or that \eqref{OO90} holds for large $t$, we obtain that for large $t$,
$$y''(t)\geq 4\int (\partial_tu(t,x))^2\,dx.$$

Using Cauchy-Schwarz inequality, \eqref{OO104} and the definition of $y(t)$, we deduce \eqref{OO101}, which concludes the proof.
 \end{proof}
\subsection{Existence of a radiation term}
We next prove:
\begin{proposition}
 \label{P:radiation}
 Let $u$ be a radial solution of \eqref{eq:NLW} such that $T_+(u)=+\infty$. Then there exists a solution $v_L$ of the linear wave equation \eqref{eq:LW} such that
 \begin{equation}
  \label{radiation}
  \forall A\in \R
  ,\quad \lim_{t\to+\infty} \int_{|x|\geq A+|t|}\left|\nabla_{t,x}(u-v_L)\right|^2\,dx=0.
 \end{equation} 
\end{proposition}
(see \cite[Lemma 3.7]{DuKeMe13} for the analog for radial solutions of the energy critical equation on $\R^3$).
\begin{proof}
 \emph{Step 1.} We prove:
 \begin{equation}
  \label{OO120} 
  \forall A\in \R,\quad \left\| \indic_{\{|x|>A+|t|\}}u\right\|_{L^{2m+1}\left( [0,+\infty),L^{2(2m+1)}(\Omega) \right)}<\infty.
 \end{equation} 
 Let $(t_n)_n$ be a sequence given by Lemma \ref{L:bnd} such that 
 \begin{equation}
  \label{OO121} 
  \lim_{n\to\infty}t_n=+\infty,\quad \limsup_{n\to\infty} \|\vec{u}(t_n)\|_{\dot{\SF}^1}<\infty.
 \end{equation} 
 By the small data theory outside wave cones and finite speed of propagation, it is sufficient to prove that for large $n$,
 \begin{equation}
 \label{OO22}
 \left\| \indic_{\{|x|\geq A+|t|\}} S_L(t-t_n)\vec{u}(t_n)\right\|_{L^{2m+1}([t_n,+\infty),L^{2(2m+1)})}\leq \eps_0,
 \end{equation} 
 where $\eps_0>0$ is a small constant given by the small data theory. Let $\left(U^j_L,\left(t_{j,n}\right)_n\right)$ be a profile decomposition for the sequence $\vec{u}(t_n)$. Without loss of generality, we can assume 
 \begin{equation}
\label{t1ntjn}
 \forall n, \quad t_{1,n}=0\text{ and } \forall j\geq 2,\; \lim_{n\to\infty}t_{j,n}\in \{\pm \infty\}.  
 \end{equation} 
Let $B\geq 1$ such that 
$$\left\|\indic_{\{|x|\geq B+|t|\}} U_L^1\right\|_{L^{2m+1}\left([0,\infty),L^{2(2m+1)}(\Omega)\right)}\leq \eps_0/2.$$
By dominated convergence, using \eqref{t1ntjn}, we have for $j\geq 2$
\begin{multline*} 
\left\|\indic_{\{|x|\geq B+|t|\}} U_L^j(\cdot-t_{j,n})\right\|_{L^{2m+1}\left([0,\infty),L^{2(2m+1)}(\Omega)\right)}\\
=\left\|\indic_{\{|x|\geq B+|t+t_{j,n}|\}} U_L^j\right\|_{L^{2m+1}\left([-t_{j,n},\infty),L^{2(2m+1)}(\Omega)\right)} \underset{n\to\infty}{\longrightarrow} 0.
\end{multline*} 
This implies that for large $n$
$$\left\|S_L(t)\vec{u}(t_n)\indic_{\{|x|\geq B+|t|\}}\right\|_{L^{2m+1}([0,\infty),L^{2(2m+1)}(\Omega)}\leq 2\eps_0/3,$$
which yields \eqref{OO22} by the small data theory.

\emph{Step 2. } We prove that for all $A\in \R$, there exists a solution $v_L^A$ of the linear wave equation \eqref{eq:LW} such that 
\begin{equation}
 \label{OO140}
 \lim_{t\to+\infty} \int_{|x|\geq A+|t|} \left|\nabla_{t,x}(u-v_L^A)\right|^2\,dx=0.
\end{equation}
Indeed, this follows immediately from Step 1, noticing that $u$ coincide, for $|x|\geq A+t$ ($t\geq 0$), with the solution $u^A$ of 
\begin{equation}
 \left\{
 \begin{aligned}
 \partial_t^2u^A-\Delta u^A&= (u^A)^{2m+1}\indic_{\{|x|\geq A+|t|\}},\quad (t,x)\in [0,\infty)\times \Omega\\
 \vec{u}^A_{\restriction t=0}&=(u_0,u_1),\quad u_{\restriction \partial \Omega}=0.
\end{aligned}\right.
 \end{equation} 
Since by Step 1 the right-hand side of the equation is in $L^1\left((0,\infty),L^2(\Omega)\right)$, we obtain the existence of $v_L^A$ satisfying \eqref{OO140}.

\emph{Step 3.} In this step we conclude the proof, proving that $v_L^A$ can be taken independent of $A$. We let $G^A$ be the unique element of $L^2(\R)$ such that 
\begin{align*}
\lim_{t\to+\infty} \int_0^{+\infty} \left|r\partial_rv_L^A-G^A(r-t)\right|^2\,dr&=0\\
\lim_{t\to+\infty} \int_0^{+\infty} \left|r\partial_tv_L^A+G^A(r-t)\right|^2\,dr&=0
\end{align*}
(see Lemma \ref{L:asymptotic}). By Lemma \ref{L:bnd}, there exists a constant $C_m$ such that 
\begin{equation}
\label{boundGA}
\|G^A\|_{L^2}^2\leq C_mE(u_0,u_1). 
\end{equation} 
By the construction of $v_L^A$ in Step 2, we have 
$$\lim_{t\to\infty}\int_{|x|\geq B+|t|} \left|\nabla_{t,x}(v_L^A-v_L^B)\right|^2\,dx=0$$
if $A\leq B$. This proves that $G^A(\eta)=G^B(\eta)$ if $\eta\geq B=\max(A,B)$. We define $G$ by
$$G(\eta)=G^{\eta-1}(\eta),$$
so that if $\eta\geq A$, $G(\eta)=G^A(\eta)$. We
note in particular that by \eqref{boundGA}, $G\in L^2(\R)$.
Let $v_L$ be the solution of \eqref{eq:LW}, given by Lemma \ref{L:asymptotic}, such that
\begin{align*}
\lim_{t\to+\infty} \int_0^{+\infty} \left|r\partial_rv_L-G(r-t)\right|^2\,dr&=0\\
\lim_{t\to+\infty} \int_0^{+\infty} \left|r\partial_tv_L+G(r-t)\right|^2\,dr&=0.
\end{align*}
Using \eqref{OO140} and the definition of $G$ and $v_L$, we obtain that $v_L$ satisfies the desired estimate \eqref{radiation}.
\end{proof}
\subsection{Proof of the soliton resolution}
\label{SS:resolution}
In this subsection we conclude the proof of Theorem \ref{T:classification}. We consider a solution $u$ of \eqref{eq:NLW}. We assume that $u$ is well defined for $t\geq 0$, and we let $v_L$ be its dispersive component, given by Proposition \ref{P:radiation}.
 
\medskip 
 
\noindent\emph{Step 1.} We prove that for all sequence $t_n\to+\infty$ such that $\vec{u}(t_n)$ is bounded in $\HHH(\Omega)$, there exists an subsequence of $(t_n)_n$ (still denoted by $(t_n)_n$), and a stationary solution $Q$ such that
\begin{equation}
 \label{resolution-sequence}
 \lim_{n\to\infty}\left\|\vec{u}(t_n)-\vec{v}_L(t_n)-(Q,0)\right\|_{\HHH(\Omega)}=0.
\end{equation} 
Let $t_n$ be such a sequence. According to Proposition \ref{P:profile}, we can assume (extracting subsequences if necessary), that the sequence $S_L(t)(\vec{u}(t_n)-\vec{v}_L(t_n))$ has a profile decomposition $\left\{U^j_L,(t_{j,n})_n\right\}_{j\geq 1}$. We assume as usual
$$ \forall j\geq 2,\; \lim_{n\to \infty} t_{j,n}\in \{\pm\infty\}\quad\text{and} \quad \forall n,\; t_{1,n}=0.$$

We note  that the solution sequence $\vec{S}_L(-t_n)(\vec{u}(t_n))$ converges weakly to $\vec{v}_L(0)$. 
Denoting by $U^0_L=v_L$, $t^0_n=-t_n$, we see that  $\left\{U^j_L,(t_{j,n})_n\right\}_{j\geq 0}$ is a profile decomposition for $S_L(t)(\vec{u}(t_n))$. In particular,
$$\forall j\geq 2,\quad \lim_{n\to\infty} |t_n-t_n^j|=+\infty.$$
We prove by contradiction 
\begin{equation}
\label{U0}
\forall j\geq 2,\quad U^j \not \equiv 0\Longrightarrow \lim_{n\to\infty} t_n-t_n^j=+\infty.
\end{equation} 
Assume on the contrary that there exists $j\geq 2$ such that
\begin{equation}
\label{V0}
U^j \not \equiv 0\text{ and } \lim_{n\to\infty} t_n^j-t_n=+\infty.
\end{equation} 
Recall
\begin{equation}
 \label{V1}
 \vec{S}_L\left( t_n^j \right)\vec{u}(t_n)\xrightharpoonup[n\to \infty]{}\vec{U}_L^j(0),\text{ weakly in }\HHH.
\end{equation} 
Let $a_n(t)=S_L(t)\vec{u}(t_n)$. By the strong Huygens principle (see the first line of \eqref{exp_u_lin_2})
\begin{equation}
 \label{V2}
 \int_{1\leq |x|\leq M} |\nabla_{t,x}a_n(t_n^j,x)|^2\,dx\leq \int_{t_n^j-M\leq |x|\leq t_n^j+M}|\nabla_{t,x}u(t_n,x)|^2\,dx.
\end{equation} 
Since by \eqref{U0},
$$\lim_{n\to\infty} \int_{|x|\geq t_n^j-M-t_n}|\nabla_{t,x}u(0,x)|^2\,dx=0,$$
we obtain (by finite speed of propagation again)
$$\lim_{n\to\infty} \int_{|x|\geq t_n^j-M}|\nabla_{t,x}u(t_n,x)|^2\,dx=0,$$
and thus \eqref{V2} implies
$$\lim_{n\to\infty}\int_{1\leq |x|\leq M}|\nabla_{t,x}a_n(t_n^j,x)|^2\,dx=0.$$
By \eqref{V1}, $U_L$ is identically $0$ , contradicting \eqref{V0}.

As usual, we denote by $U^1$ the solution of \eqref{eq:NLW} with initial data $U^1(0)$. 
We next prove that $U^1$ is a stationary solution. If not, by Proposition \ref{P:rigidity}, there exists $R\geq 1$ such that $U^1$ is well-defined for $\{|x|>R+|t|\}$, and 
\begin{equation}
\label{U1}
\sum_{\pm\infty} \lim_{t\to\pm \infty} \int_{\{|x|>R+|t|\}} \left|\nabla_{t,x}U^1(t,x)\right|^2\,dx>0.
\end{equation}
We let 
$$w_n^J(t)=u_L(t+t_n)-v_L(t+t_n)-\sum_{j=1}^J U^j_L(t-t_{j,n}),$$
and
\begin{equation}
\label{epsilon_n_J}
\epsilon_n^J(t)=u(t+t_n)-v_L(t+t_n)-U^1(t)-\sum_{j=2}^J U^j_L(t-t_{j,n})-w_n^J(t). 
\end{equation} 
By Proposition \ref{P:approx_profile}, $u(t_n+t)$ is well defined for $\{|x|>R+|t|\}$, and 
$$\lim_{J\to\infty}\limsup_{n\to\infty}\left(\sup_{t\in \R} \left\|\indic_{\{|x|>R+|t|\}} \nabla_{t,x}\epsilon_n^J(t)\right\|_{L^2}\right)=0.$$
We first consider the case where
\begin{equation}
\label{U1'}
\lim_{t\to+ \infty} \int_{\{|x|>R+|t|\}} \left|\nabla_{t,x}U^1(t,x)\right|^2\,dx=\eta_+>0.
\end{equation}
By \eqref{epsilon_n_J}, for all $t\geq 0$,
\begin{multline}
\label{scal_product}
\int_{|x|>R+|t|} \left(\nabla_{t,x}(u-v_L)(t+t_n)\right)\cdot\nabla_{t,x} U^1(t)\,dx
=
\int_{|x|>R+|t|} |\nabla_{t,x} U^1(t)|^2\,dx\\
+\sum_{j=2}^J \int_{|x|>R+|t|} \nabla_{t,x}U^j_L(t-t_{j,n})\cdot\nabla_{t,x}U^1(t)\,dx\\
+ \int_{|x|>R+|t|} \nabla_{t,x}w_n^J(t)\cdot \nabla_{t,x} U^1(t)\,dx+o_n(1),
\end{multline}
where $o_n(1)$ goes to $0$ as $n$ goes to infinity, uniformly with respect to $t\geq 0$.
Using that $\lim_{n\to\infty} |t_n^j-t_n^k|=+\infty$ for $j\neq k$ and the property \eqref{dispersive_part} of $w_n^J$, it is easy to prove that lines $2$ and $3$ of \eqref{scal_product} go to $0$ as $n\to\infty$
(see e.g. Claim 3.2 in \cite{DuKeMe13}), and thus, by \eqref{U1'}, for large $n$, 
$$\lim_{t\to\infty}
\int_{|x|>R+|t|} \left|\nabla_{t,x}(u-v_L)(t+t_n,x)\right|^2\,dx\geq \eta_+/2.$$
In other words, for large $n$,
$$\lim_{t\to\infty}
\int_{|x|>R+t-t_n} \left|\nabla_{t,x}(u-v_L)(t,x)\right|^2\,dx\geq \eta_+/2,$$
which contradicts the definition of $v_L$ given by Proposition \ref{P:radiation}.

We next assume 
\begin{equation}
\label{U1''}
\lim_{t\to- \infty} \int_{\{|x|>R+|t|\}} \left|\nabla_{t,x}U^1(t,x)\right|^2\,dx=\eta_->0.
\end{equation}
Arguing as before, we obtain that for large $n$, using the analog of \eqref{scal_product} with $t=-t_n$
$$
\int_{|x|>R+t_n} \left|\nabla_{t,x}u(0,x)\right|^2\,dx\geq \eta_-/2.$$
Since $t_n$ is arbitrarily large, we obtain a contradiction, proving that $U^1$ is a stationary solution $Q$. Note that the case $Q\equiv 0$ is not excluded. In any case, we have, by explicit computation:
$$\indic_{\{|x|>|t|\}}Q\in {L^{2m+1}\left(\R,L^{2(2m+1)}\right)}, $$
so that the assumptions of Proposition \ref{P:approx_profile} (and its analog in the past) are satisfied with $R=1$. As a consequence, letting
$$ \epsilon_n^J(t,x)=u(t+t_n,x)-v_L(t+t_n,x)-Q(x)-\sum_{j=2}^J U^j_{L}(t-t_{j,n},x)-w_n^J(t,x),$$
we have 
\begin{equation}
\label{remainder_nul}
\lim_{J\to\infty} \limsup_{n\to\infty} \sup_{t\in \R} \int_{|x|>|t|+1}\left|\nabla_{t,x}\epsilon_n^J(t,x)\right|^2\,dx=0. 
\end{equation} 

We next prove by contradiction that $U^j_L \equiv 0$ for $j\geq 2$. Assume that there exists $j\geq 2$ such that $U^j_L$ is not zero. Then by Lemma \ref{L:asymptotic}, we have, for large $A$,
\begin{equation}
 \label{asymptotic_Uj}
 \lim_{t\to\pm\infty} \int_{|t|-A<|x|<|t|+A} \left|\nabla U^j_L(t,x)\right|^2\,dx=\eta_{\pm}>0.
\end{equation} 
First assume
$$\lim_{n\to\infty} t_{j,n}=-\infty.$$
Combining \eqref{remainder_nul}, \eqref{asymptotic_Uj} and the pseudo-orthogonality of the time sequences $(t_{j,n})_n$, we can obtain that for a large fixed $n$, 
$$\lim_{t\to+\infty} \int_{t-t_n-t_{j,n}-A\leq |x|\leq t-t_n-t_{j,n}+A}|\nabla_{t,x}(u-v_L)|^2\,dx \geq \eta_+/2.$$
This contradicts the definition of $v_L$ in Proposition \ref{P:radiation}
Next assume
$$\lim_{n\to\infty} t_{j,n}=+\infty.$$
Using that by \eqref{remainder_nul},
$$   \lim_{J\to+\infty} \limsup_{n\to\infty} \sup_{t\in \R} \int_{|x|>|t_{n}|+1}\left|\nabla_{t,x}\epsilon_n^J(t-t_{n},x)\right|^2\,dx=0,$$
we obtain that for all large $n$,
\begin{equation}
 \label{asymptotic_Uj_bis}
 \int_{t_n+t_{j,n}-A<|x|<t_n+t_{j,n}+A} \left|\nabla_{t,x}u(0)\right|^2\,dx\geq \frac{\eta_{-}}{2},
\end{equation} 
a contradiction, since $\vec{u}(0)\in \HHH(\Omega)$. 

Since $U^j_L\equiv 0$ for $j\geq 2$, we see that $w_n^J$ and $\epsilon_n^J$ do not depend on $J\geq 2$. We will denote $w_n=w_n^J$ and $\eps_n=\eps_n^J$. We are left with proving 
$$\lim_{n\to\infty} \|\vec{w}_n(0)\|_{\HHH(\Omega)}=0.$$
Since by Lemma \ref{L:exterior_energy}, 
\begin{equation}
\label{wn}
\sum_{\pm}\lim_{t\to\pm\infty} \int_{|x|\geq |t|+1}|\nabla_{t,x}w_n|^2\,dx\geq \frac{1}{2}\|\vec{w}_n(0)\|_{\HHH(\Omega)}^2, 
\end{equation} 
we can deduce, with the same arguments as before,
$$\lim_{t\to+\infty} \int_{|x|\geq t+t_n}|\nabla_{t,x}(u-v_L)|^2\,dx \geq \frac{1}{2}\left\|\vec{w}_n(0)\right\|_{\HHH}^2,$$
if \eqref{wn} holds for large $n$ with a sign $+$, and 
\begin{equation*}
 \int_{|x|>t_n} \left|\nabla_{t,x}u(0)\right|^2\,dx\geq \frac{1}{2}\left\|\vec{w}_n(0)\right\|_{\HHH}^2,
\end{equation*} 
if \eqref{wn} holds for large $n$ with a sign $-$. This yields, in both cases, a contradiction, concluding this step.

\medskip

\noindent\emph{Step 2. Conclusion of the proof.}
Let $t_n\to+\infty$ be as in the preceding step. In view of \eqref{resolution-sequence}, we must prove
\begin{equation}
 \label{full-resolution}
\lim_{t\to\infty}\left\|\vec{u}(t)-\vec{v}_L(t)-(Q,0)\right\|_{\HHH(\Omega)}=0.
\end{equation} 
We assume that \eqref{full-resolution} does not hold, and fix a small $\eps>0$, such that 
$$ \limsup_{t\to\infty}\left\|\vec{u}(t)-\vec{v}_L(t)-(Q,0)\right\|_{\HHH(\Omega)}>\eps.$$
Let 
$$t_n'=\min\Big\{t>t_n\text{ s.t. } \left\|\vec{u}(t)-\vec{v}_L(t)-(Q,0)\right\|_{\HHH(\Omega)}>\eps \Big\},$$
so that $t_n<t_n'$ and 
\begin{equation}
 \label{dist_eps}
\left\|\vec{u}(t'_n)-\vec{v}_L(t'_n)-(Q,0)\right\|_{\HHH(\Omega)}=\eps. 
\end{equation} 
By Step 1, there exists a stationary solution $Q'$ such that 
\begin{equation}
 \label{res_seq'}
 \lim_{n\to\infty}
\left\|\vec{u}(t'_n)-\vec{v}_L(t'_n)-(Q',0)\right\|_{\HHH(\Omega)}=0. 
\end{equation}
By the triangle inequality, \eqref{dist_eps} and \eqref{res_seq'},
\begin{equation}
 \label{Q_Q'}
 \left\|Q-Q'\right\|_{\HHH(\Omega)}\leq \eps.
\end{equation} 
By \eqref{resolution-sequence}, and the conservation of the linear and the nonlinear energy:
$$E(Q,0)+\frac 12 \|\vec{v}(0)\|_{\HHH(\Omega)}^2=E(u_0,u_1).$$
Similarly, by \eqref{res_seq'},
$$E(Q',0)+\frac 12 \|\vec{v}(0)\|_{\HHH(\Omega)}^2=E(u_0,u_1).$$
This proves that 
$$E(Q,0)=E(Q',0).$$
By the classification of the radial stationary solutions in Subsection \ref{SS:stationary}, we obtain that $Q=Q'$, or $Q\neq 0$ and $Q=-Q'$. 
The first case contradicts \eqref{dist_eps} or \eqref{res_seq'}. In the second case $\|Q-Q'\|_{\HHH}=2\|Q\|_{\HHH}\geq 2\|Q_0\|_{\HHH}$, where $Q_0$ is the ground state (see  Subsection \ref{SS:stationary}). This contradicts  \eqref{Q_Q'} if $\eps$ is small enough. The proof is complete.
\begin{remark}
 \label{R:proof_defocusing}
 Proposition \ref{P:radiation} (exitence of a radiation term $v_L$) is still valid with the same proof, for the defocusing analog of \eqref{eq:NLW}.
 If $u$ is a solution of the defocusing analog of \eqref{eq:NLW}, then Remark \ref{R:rigidity_defocusing},  and Step 1 of the preceding proof yield
 the existence of a sequence $t_n\to+\infty$ such that 
 $$\lim_{n\to\infty}\|\vec{u}(t_n)-\vec{v}_L(t_n)\|_{\HHH(\Omega)}= 0.$$
 This implies, by the small data well-posedness theory that 
 $$\|u\|_{L^{2m+1}\left(L^{2(2m+1)}(\Omega)\right) }<\infty$$
 for large $n$, and thus that $u$ scatters.
 \end{remark}

\section{Further elements on the dynamics}
\label{S:further}
\subsection{Dynamics below the energy threshold}
\label{SS:threshold}
In this section we prove Corollary \ref{cor:ground_state}. 

Let $(u_0,u_1)\in \HHH$ with $E(u_0,u_1)\leq E(Q_0,0)$, and denote by $(T_-,T_+)$ its maximal interval of existence.

We start by variational considerations.
Using the  Sobolev inequality of Proposition \ref{prop:minimizer}, the fact that $\int_{\Omega} |\nabla Q_0|^2=\int_{\Omega} Q_0^{2m+2}$, and the conservation of the energy we obtain
\begin{equation}
 \label{FromSobolev}
 E(Q_0,0)\geq E(u_0,u_1)\geq f\left( |\nabla u(t)|^2\right)+\frac{1}{2}\left\|\partial_tu(t)\right\|_{L^2}^2,
\end{equation} 
where $f(\sigma)=\frac{\sigma}{2}-\frac{1}{2m+2} \frac{\sigma^{m+1}}{\left(\int_{\Omega}|\nabla Q_0|^2\right)^m} $. The function $f$ is increasing on $\left(0,\int_{\Omega} |\nabla Q_0|^2\right)$, decreasing on $\left( \int_{\Omega}|\nabla Q_0|^2,+\infty \right)$ and satisfies $f\left(\int_{\Omega}|\nabla Q_0|^2\right)=E(Q_0,0)$. In parti\-cular, $E(Q_0,0)$ is the maximum of $f$ and it is attained at $\sigma=\int |\nabla Q_0|^2$. We deduce from \eqref{FromSobolev} that for all $t\in (T_-,T_+)$
$$ \int_{\Omega} |\nabla u(t)|^2=\int_{\Omega}|\nabla Q_0|^2\Longrightarrow \int_{\Omega}(\partial_tu(t))^2=0\text{ and } E(\vec{u}(t))=E(Q_0,0).$$
Thus if $\int_{\Omega}|\nabla u(t)|^2=\int_{\Omega}|\nabla Q_0|^2$, for one $t\in (T_-,T_+)$, we must have  $\int_{\Omega}|u(t)|^{2m+2}=\int_{\Omega}|Q_0|^{2m+2}$, and the uniqueness in Proposition \ref{prop:minimizer} shows that $\vec{u}(t)=\pm (Q_0,0)$, and thus that $u$ is a stationary solution. By the intermediate value theorem,
\begin{gather}
\label{sub}
\int_{\Omega}|\nabla u_0|^2<\int_{\Omega}|\nabla Q_0|^2\Longrightarrow \forall t\in (T_-,T_+),\; \int_{\Omega}|\nabla u(t)|^2<\int_{\Omega}|\nabla Q_0|^2\\
\label{super}
\int_{\Omega}|\nabla u_0|^2>\int_{\Omega}|\nabla Q_0|^2\Longrightarrow \forall t\in (T_-,T_+),\; \int_{\Omega}|\nabla u(t)|^2>\int_{\Omega}|\nabla Q_0|^2.
\end{gather} 

\medskip

\noindent\emph{Case 1: global existence.}
Assume that we are in the case where the left-hand side of \eqref{sub} is satisfied. We see that $u$ is bounded in $\dot{H}^{1}(\Omega)$, and thus, by conservation of the energy, that $\vec{u}$ is bounded in $\HHH$. Thus $u$ is global. 

Furthermore, Theorem \ref{T:classification} and the condition $E(u_0,u_1)\leq E(Q_0,0)$ implies that if $u$ does not scatter forward (respectively backward) in time to a linear solution, then
\begin{equation}
\label{goingtoQ}
\lim_{t\to +\infty} \|\vec{u}(t)-(Q_0,0)\|_{\HHH}=0 
\end{equation} 
(respectively $\lim_{t\to-\infty}\ldots$). However we see by Proposition \ref{P:rigidity} that both properties cannot occur simultaneously, i.e. that $u$ must scatter in at least one time direction.

\medskip

\noindent\emph{Case 2: finite time blow-up.} Next, we assume that we are in the case where the left-hand side of \eqref{super} is satisfied.
Note that if $u$ is global and scatters to a linear solution, say forward in time, then we must have
$$\lim_{t\to+\infty}\frac{1}{2}\int_{\Omega}|\nabla_{x,t} u(t)|^2=E(u_0,u_1)\leq E(Q_0,0)<\frac{1}{2}\int |\nabla Q_0|^2.$$
Thus \eqref{super} implies that $u$ cannot scatter to a linear solution in any time direction. As a consequence, if $T_+=+\infty$, then by Theorem \ref{T:classification} \eqref{goingtoQ} must be satisfied and similarly for negative times. Again, Proposition \ref{P:rigidity} implies that both properties cannot occur simultaneously, which concludes the proof. \qed

\subsection{One-pass theorem}
\label{SS:one-pass}
In this subsection we prove Theorem \ref{T:one-pass}. Denote by $\Sigma=\{0\}\cup \bigcup_k\{Q_k\}\cup \bigcup_k\{-Q_k\}$ the set of stationary solutions. We argue by contradiction, assuming there there exist $\eps>0$, and, for all $n\geq 1$, $s_{n}<t_{n}'<t_{n}$, a solution $u_n$ of \eqref{eq:NLW} defined on $[s_{n},t_{n}]$ and such that 
\begin{gather}
 \label{OP01}
\lim_{n\to\infty}\left(\left\|\vec{u}_n(s_{n})-(Q_k,0)\right\|_{\HHH}+\min_{Q\in \Sigma}\left\|\vec{u}_n(t_{n})-(Q,0)\right\|_{\HHH}\right)=0\\
\label{OP02}
\forall n,\quad \left\|\vec{u}_n(t'_n)-(Q_k,0)\right\|_{\HHH}\geq \eps.
\end{gather} 
By the intermediate value theorem, we can replace the inequality in \eqref{OP02} by an equality. Translating in time, we can assume $t_n'=0$. Furthermore, by energy conservation, we can replace the minimum in \eqref{OP01} by $\left\|\vec{u}_n(t_{n})-\iota (Q_k,0)\right\|_{\HHH}$ for some sign $\iota\in \{\pm 1\}$. Thus we can replace \eqref{OP01} and \eqref{OP02} by
\begin{gather}
 \label{OP01'}
\lim_{n\to\infty}\Bigl(\left\|\vec{u}_n(s_{n})-(Q_k,0)\right\|_{\HHH}+\left\|\vec{u}_n(t_{n})-\iota(Q_k,0)\right\|_{\HHH}\Bigr)=0\\
\label{OP02'}
\forall n,\quad \left\|\vec{u}_n(0)-(Q_k,0)\right\|_{\HHH}= \eps,
\end{gather} 
where $s_n<0<t_n$. Extracting subsequences if necessary, we consider a profile decomposition $\left\{U^j_L,(t_{j,n})_n\right\}_{j\geq 1}$ of $\vec{u}_n(0)$. As in Subsection \ref{SS:profile}, we assume
$$ \forall n,\quad t_{1,n}=0,\quad j\geq 2\Longrightarrow \lim_{n\to\infty}t_{j,n}\in \{\pm\infty\}.$$
By \eqref{OP02'} and the Pythagorean expansion of the $\HHH$ norm, we have
\begin{equation}
 \label{OP11}
\left\|\vec{U}^1_L(0)-(Q_k,0)\right\|_{\HHH}\leq \eps.
\end{equation} 
We distinguish two cases.

If $\vec{U}^1_L(0)=(Q_k,0)$, then \eqref{OP02'} and the Pythagorean expansion of the energy show that 
$$\lim_{n\to\infty} E(\vec{u}_n(0))>E(Q_k,0),$$
a contradiction with \eqref{OP01'}. 

If $\vec{U}^1_L(0)\neq (Q_k,0)$, then by \eqref{OP11} and the classification of stationary solutions (Proposition \ref{P:stationary}), since $\eps$ is small, we see that $\vec{U}^1_L(0)$ is not a stationary solution. By \eqref{OP11}, we also now (using again that $\eps$ is small) that the solution $U^1$ of \eqref{eq:NLW} with initial data $\vec{U}^1_L(0)$ is well-defined on $\{r>|t|+1\}$.  As a consequence, by Proposition \ref{P:rigidity}, $U^1$ satisfies:
$$ \sum_{\pm} \lim_{t\to \pm\infty} \int_{|x|>|t|+1}|\nabla U^1(t,x)|^2\,dx>0$$
By the small data well-posedness theory, this implies 
\begin{equation}
\label{infimum}
\inf_{t\geq 0}\int_{|x|\geq |t|+1} |\nabla U^1(t,x)|^2\,dx+\inf_{t\leq 0}\int_{|x|\geq |t|+1} |\nabla U^1(t,x)|^2\,dx>0.
\end{equation} 
Thus there is a small $\eta>0$ such that the following holds for all large $n$:
\begin{equation}
\label{lowerbnd}
\int_{|x|>|\sigma_n|+1}|\nabla u_n(\sigma_n,x)|^2\,dx>\eta,
\end{equation}  
where $\sigma_n=s_n$ if the infimum  for $t\leq 0$ in \eqref{infimum} is positive, and $\sigma_n=t_n$
if the infimum for $t\geq 0$ is positive.

Arguing as in Subsection \ref{SS:resolution}, we deduce that the following holds for all $n$:
$$ \int_{|x|>|\sigma_n|+1}|\nabla u_n(\sigma_n,x)|^2\,dx>\frac{\eta}{2}.$$
Combining with \eqref{OP01'} we deduce that for large $n$
$$ \int_{|x|>R+|\sigma_n|}|\nabla Q_k(x)|^2\,dx>\frac{\eta}{4}.$$
This is a contradiction since by \eqref{OP01'} and \eqref{OP02'} and the continuity of the flow for equation \eqref{eq:NLW}, we must have $\lim_{n\to\infty} |\sigma_n|=+\infty$.

\bibliographystyle{acm}
\bibliography{toto}
\end{document}